\documentclass[10pt,twoside]{siamart1116}
\usepackage[english]{babel}
\usepackage{graphicx,epstopdf,epsfig}
\usepackage{amsfonts,epsfig,fancyhdr,graphics, hyperref,amsmath,amssymb}

\newcommand{\HE}{Namie of Handling Editor}
\newcommand{\DoS}{Month/Day/Year}
\newcommand{\DoA}{Month/Day/Year}
\newcommand{\CA}{Name of Corresponding Author}

\newcommand{\Names}{Alyssa Adams and Bonnie Jacob}
\newcommand{\Title}{Failed zero forcing and critical sets on directed graphs}

\newtheorem{dfn}[theorem]{Definition}
\newtheorem{prop}[theorem]{Proposition}
\newtheorem{obs}[theorem]{Observation}

\renewcommand{\theequation}{\thesection.\arabic{equation}}
\renewtheorem{theorem}{Theorem}[section]

\usepackage{tikz}
\usetikzlibrary{decorations}
\usetikzlibrary{decorations.pathmorphing}
\usetikzlibrary{decorations.pathreplacing}
\usetikzlibrary{decorations.markings}
\usetikzlibrary{decorations.text}
\usetikzlibrary{arrows.meta}
\usetikzlibrary{calc}
\usetikzlibrary[arrows,positioning, shapes, fit]
\usetikzlibrary{snakes}
\usepackage{tkz-euclide}

\newcommand{\Z}{\operatorname{Z}}
\newcommand{\F}{\operatorname{F}}

\newcommand{\og}[1]{{\overrightarrow{#1}}}
\newcommand{\ogg}{\overrightarrow{G}}

\begin{document}

\bibliographystyle{plain}

%  Leave these commented lines here
%\input{ELAheader-template.tex}
% ELA insert correct page number
\setcounter{page}{1}

\thispagestyle{empty}

%Insert the title of the paper
 \title{\Title\thanks{Received
 by the editors on \DoS.
 Accepted for publication on \DoA. 
 Handling Editor: \HE. Corresponding Author: \CA}}

\author{
Bonnie Jacob\thanks{Department of Science and Mathematics,
National Technical Institute for the Deaf, Rochester Institute of Technology, 14623, USA
(bcjntm@rit.edu).}
% Remember to put \and between any two authors
\and
Alyssa Adams\thanks{Department of Mathematics and Statistics, Youngstown State University, Youngstown, OH 44503, USA (anadams02@student.ysu.edu).  Supported by the National Science Foundation under Grant No. 1659299. }
}

\markboth{\Names}{\Title}

\maketitle

%%%%%%%%%
%%%% OLD starts here
%%%%%%%%

\begin{abstract}
Let $D$ be a simple digraph (directed graph) with vertex set $V(D)$ and arc set $A(D)$ where $n=|V(D)|$, and each arc is an ordered pair of distinct vertices.  If $(v,u) \in A(D)$, then  $u$ is considered an \emph{out-neighbor} of $v$ in $D$.  Initially, we designate each vertex to be either filled or empty.  Then, the following color change rule (CCR) is applied: if a filled vertex $v$ has exactly one empty out-neighbor $u$, then $u$ will be filled. The process continues until the CCR does not allow any empty vertex to become filled. If all vertices in $V(D)$ are eventually filled, then the initial set is called a \emph{zero forcing set} (ZFS); if not, it is a \emph{failed zero forcing set} (FZFS).  We introduce the \emph{failed zero forcing number} $\F(D)$ on a digraph, which is the maximum cardinality of any FZFS.  The \emph{zero forcing number}, $\Z(D)$, is the minimum cardinality of any ZFS.  We characterize digraphs that have $\F(D)<\Z(D)$ and determine $\F(D)$ for several classes of digraphs including directed acyclic graphs, weak paths and cycles, and weakly connected line digraphs such as de Bruijn and Kautz digraphs.  We also characterize digraphs with  $\F(D)=n-1$, $\F(D)=n-2$, and $\F(D)=0$, which leads to a characterization of digraphs in which any vertex is a ZFS.  Finally, we show that for any integer $n \geq 3$ and any non-negative integer $k$ with $k <n$, there exists a weak cycle $D$ with $\F(D)=k$.
\end{abstract}

\begin{keywords}
Zero forcing, line digraph, critical set
\end{keywords}

\begin{AMS}
05C50, 15A03
\end{AMS}

\section{Introduction}

In this paper, we study failed zero forcing on simple digraphs (directed graphs).  
Zero forcing problems, including failed zero forcing, are based on a color change rule (CCR) applied to an initial coloring on the vertex set, where there are only two colors: filled or empty.  The CCR is: if a filled vertex has exactly one empty out-neighbor, then the out-neighbor will change from empty to filled.  In \cite{berliner2017zero}, the authors relate this to rumor spreading: if Astrid knows a secret, and all of Astrid's friends except Zoe know the secret, then Astrid will share the secret with Zoe.  The zero forcing number is the smallest number of vertices that initially must be filled in order for all vertices in the digraph to eventually be filled.  

There has been a great deal of work on determination of the zero forcing number \cite{barioli2010zero, berliner2017zero, berliner2015path,  burgarth2013zero, aim2008zero}.  A related question has also been studied for finite simple graphs: what is the largest number of vertices that initially could be filled, yet never lead to the entire graph being filled?  In the context of the rumor example, how many people could initially know the secret, yet the secret never spread to all the people in the network?  This is called the failed zero forcing number of a graph, and has been studied for finite simple graphs \cite{ansill2016failed, fetcie2014failed}.
The problem of computing the failed zero forcing number has been shown to be NP-hard \cite{shitov2017complexity}.  Zero forcing was studied for digraphs in \cite{berliner2017zero, berliner2015path}.
In this paper, we expand the study of failed zero forcing to digraphs, including oriented graphs.

\subsection{Definitions and notation}

We denote by $D=(V,A)$ a finite simple digraph with vertex set $V$ where $n=|V|$ and arc set $A$, or $V(D)$ and $A(D)$ respectively in the case the digraph in question is ambiguous. We primarily use digraph notation based on \cite{bang2008digraphs}. The word \emph{simple} indicates that the digraph has no loops (that is, no arcs of the form $(u,u)$) or more than one copy of any arc (no multiple or parallel arcs), where an arc $(u,v)$ is an ordered pair of vertices with \emph{tail} $u$ and \emph{head} $v$.  Note that $(u,v), (v,u) \in A$ is permitted, since the head and tail of each is swapped.   The \emph{complement} of $D$, which we denote by $\overline{D}$, is a digraph such that $V(\overline{D}) = V(D)$, and for any $u, v \in V(\overline{D})$, $(u,v) \in A(\overline{D})$ if and only if $(u,v) \notin A(D)$.    Some digraphs with loops are investigated in Section \ref{sec:line}.  An \emph{oriented graph} $D$ is a digraph with no cycle of length 2.  That is, if $D$ is an oriented graph with $(u,v) \in A(D)$, then $(v,u) \notin A(D)$.  We use $uv$ in place of $(u,v)$ throughout the paper.  For any $S \subseteq V$, we refer to $|S|$ as the \emph{order} of  $S$ and to $|V|$ as the \emph{order} of the digraph.

For a vertex $u\in V(D)$, the \emph{open in-neighborhood} of $u$ in $D$, denoted $N_D^-(u)$, is $N_D^-(u)=\{v \in V : vu  \in A\}$.  The \emph{closed in-neighborhood} of $u$, denoted $N_D^-[u]$, is the set $N_D^-(u) \cup \{u\}$.    The \emph{open out-neighborhood} of $u$  is the set $N_D^+(u)=\{v \in V(D) : uv \in A\}$.  The \emph{closed out-neighborhood} of $u$ is the set $N_D^+[u]=N_D^+(u) \cup \{u\}$.    The \emph{in-degree} and \emph{out-degree} of $u \in V(D)$ are given by $\deg_D^-(u)=|N_D^-(u)|$ and $\deg_D^+(u)=|N_D^+(u)|$ respectively.  For $S \subseteq V(D)$, we use $N_D^+(S)$, $N_D^+[S]$, $N_D^-(S)$, and $N_D^-[S]$ to denote the respective neighborhoods.  If the digraph $D$ is understood, at times we omit mention of $D$, using $\deg^-(u)$ instead of $\deg_D^-(u)$, for example. 
A \emph{source} is a vertex $v \in V(D)$ such that $\deg_D^-(v)=0$.  A \emph{sink} is a vertex $v \in V(D)$ such that $\deg_D^+(v)=0$.

We describe zero forcing formally as follows on a simple digraph $D$.   Let $S \subseteq V$, and let $i \in \{0, 1, 2, 3, \ldots\}$.  Then
\begin{itemize}
\item $B^0(S) := S$
\item $B^{i+1}(S) := B^i(S) \cup \{w : \{w\} = N^+(v) \backslash B^i(S) \mbox{ for some } v \in B^i(S)\}$
\end{itemize}
Note that for any $i \geq 0$, $B^i(S) \subseteq B^{i+1}(S)$, and there exists some $j \geq 0$ such that $B^k(S) = B^j(S)$ for every $k \geq j$.  This formal definition is equivalent to the definition described in terms of the CCR. If $u \in B^i(S)$ for some $i$, and $\{v\}=N^+[u] \backslash B^i(S)$ (that is, if the CCR dictates that $v$ will be filled in the next iteration), we refer to this as a \emph{color change}.

\begin{dfn}
We say that $S$ is a 
\begin{itemize}
\item \emph{zero forcing set} (ZFS) if $B^t(S) = V$ for some $t \geq 0$
\item \emph{failed zero forcing set} (FZFS) otherwise.
\end{itemize}
\end{dfn}

The \emph{zero forcing number} $\Z(D)$ is the smallest order of any ZFS of $D$.  The \emph{failed zero forcing number} $\F(D)$ is the largest order of any FZFS of $D$.   If $S\subseteq V$ is a set with $|S|=\F(D)$, then we say that $S$ is a \emph{maximum} FZFS.  
  If $B^1(S)=B^0(S)$, we say that $S$ is \emph{stalled}.  Note that any maximum FZFS  is stalled.    
 The concept of a stalled zero forcing set was introduced in the context of failed skew zero forcing in \cite{ansill2016failed}.  In \cite{ferrero2019zero}, the authors introduced the idea of a critical set. 
 \begin{dfn}
A nonempty set $W \subseteq V(D)$ is called \emph{(weakly) critical} if for every $v \in V(D) \backslash W$, $|N_D^+(v) \cap W| \neq 1$, and \emph{strongly critical} if for every  $v \in V(D)$, $|N_D^+(v) \cap W| \neq 1$.  
\end{dfn}
Note that $W$ is a critical set in $D$ if and only if $V(D) \backslash W$ is stalled, and that every strongly critical set is a critical set.  Figure \ref{fig:simpleexample} shows examples of a ZFS, a FZFS that is stalled, a critical set, and a FZFS that is not stalled.  Note that for the digraph $D$ shown, neither FZFS is a maximum FZFS, since we can see that $\F(D)=n-1$ by letting $S= V \backslash \{ v \}$, for example.  
  Throughout the paper, we use the relationship between stalled sets and critical sets to establish results about FZFS as well as about critical sets. 
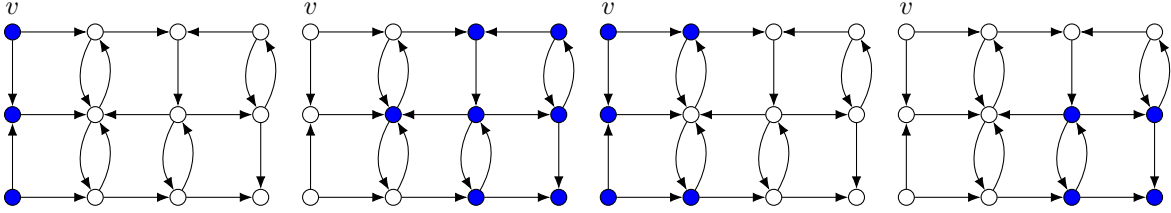
\begin{figure}[htbp]
\begin{center}
\begin{minipage}{0.24\textwidth}
\begin{tikzpicture}[auto]
\tikzstyle{vertex}=[draw, circle, inner sep=1.5mm, scale=0.5]
\tikzset{edge/.style = {->,> = Latex}}
\tikzset{curvededge/.style = {->,> = Latex, bend right}}
\node (v1) at (0,0) [vertex, fill=blue, label=$v$] {};
\node (v2) at (0,-1.1) [vertex, fill=blue] {};
\node (v3) at  (0,-2.2) [vertex, fill=blue] {};
\node (v4) at (1.1,0) [vertex] {};
\node (v5) at (1.1,-1.1) [vertex] {};
\node (v6) at  (1.1,-2.2) [vertex] {};
\node (v7) at (2.2,0) [vertex] {};
\node (v8) at (2.2,-1.1) [vertex] {};
\node (v9) at  (2.2,-2.2) [vertex] {};
\node (v10) at (3.3,0) [vertex] {};
\node (v11) at (3.3,-1.1) [vertex] {};
\node (v12) at  (3.3,-2.2) [vertex] {};
\foreach[evaluate=\y using int(\x+1)] \x in {1, 7, 11}
\draw [edge] (v\x) to (v\y);
\foreach[evaluate=\y using int(\x-1)] \x in {3}
\draw [edge] (v\x) to (v\y);
\foreach[evaluate=\y using int(\x+3)] \x in {1, 3, 4, 9}
\draw [edge] (v\x) to (v\y);
\foreach[evaluate=\y using int(\x-3)] \x in {8, 10}
\draw [edge] (v\x) to (v\y);
\foreach[evaluate=\y using int(\x-1)] \x in {5, 6, 9, 11}{
\draw [curvededge] (v\x) to (v\y);
\draw [curvededge] (v\y) to (v\x);
}
\foreach[evaluate=\y using int(\x+3)] \x in {2, 6, 8}
\draw [edge] (v\x) to (v\y);
\end{tikzpicture}
\end{minipage}
\begin{minipage}{0.24\textwidth}
\begin{tikzpicture}[auto]
\tikzstyle{vertex}=[draw, circle, inner sep=1.5mm, scale=0.5]
\tikzstyle{vertexemp}=[draw, circle, inner sep=0.8mm]
\tikzset{edge/.style = {->,> = Latex}}
\tikzset{curvededge/.style = {->,> = Latex, bend right}}
\node (v1) at (0,0) [vertex, label=$v$] {};
\node (v2) at (0,-1.1) [vertex] {};
\node (v3) at  (0,-2.2) [vertex] {};
\node (v4) at (1.1,0) [vertex] {};
\node (v5) at (1.1,-1.1) [vertex, fill=blue] {};
\node (v6) at  (1.1,-2.2) [vertex] {};
\node (v7) at (2.2,0) [vertex, fill=blue] {};
\node (v8) at (2.2,-1.1) [vertex,fill=blue] {};
\node (v9) at  (2.2,-2.2) [vertex, fill=blue] {};
\node (v10) at (3.3,0) [vertex, fill=blue] {};
\node (v11) at (3.3,-1.1) [vertex, fill=blue] {};
\node (v12) at  (3.3,-2.2) [vertex, fill=blue] {};
\foreach[evaluate=\y using int(\x+1)] \x in {1, 7, 11}
\draw [edge] (v\x) to (v\y);
\foreach[evaluate=\y using int(\x-1)] \x in {3}
\draw [edge] (v\x) to (v\y);
\foreach[evaluate=\y using int(\x+3)] \x in {1, 3, 4, 9}
\draw [edge] (v\x) to (v\y);
\foreach[evaluate=\y using int(\x-3)] \x in {8, 10}
\draw [edge] (v\x) to (v\y);
\foreach[evaluate=\y using int(\x-1)] \x in {5, 6, 9, 11}{
\draw [curvededge] (v\x) to (v\y);
\draw [curvededge] (v\y) to (v\x);
}
\foreach[evaluate=\y using int(\x+3)] \x in {2, 6, 8}
\draw [edge] (v\x) to (v\y);
\end{tikzpicture}
\end{minipage}
\begin{minipage}{0.24\textwidth}
\begin{tikzpicture}[auto]
\tikzstyle{vertex}=[draw, circle, inner sep=1.5mm, scale=0.5]
\tikzset{edge/.style = {->,> = Latex}}
\tikzset{curvededge/.style = {->,> = Latex, bend right}}
\node (v1) at (0,0) [vertex, fill=blue, label=$v$] {};
\node (v2) at (0,-1.1) [vertex, fill=blue] {};
\node (v3) at  (0,-2.2) [vertex, fill=blue] {};
\node (v4) at (1.1,0) [vertex, fill=blue] {};
\node (v5) at (1.1,-1.1) [vertex] {};
\node (v6) at  (1.1,-2.2) [vertex, fill=blue] {};
\node (v7) at (2.2,0) [vertex] {};
\node (v8) at (2.2,-1.1) [vertex] {};
\node (v9) at  (2.2,-2.2) [vertex] {};
\node (v10) at (3.3,0) [vertex] {};
\node (v11) at (3.3,-1.1) [vertex] {};
\node (v12) at  (3.3,-2.2) [vertex] {};
\foreach[evaluate=\y using int(\x+1)] \x in {1, 7, 11}
\draw [edge] (v\x) to (v\y);
\foreach[evaluate=\y using int(\x-1)] \x in {3}
\draw [edge] (v\x) to (v\y);
\foreach[evaluate=\y using int(\x+3)] \x in {1, 3, 4, 9}
\draw [edge] (v\x) to (v\y);
\foreach[evaluate=\y using int(\x-3)] \x in {8, 10}
\draw [edge] (v\x) to (v\y);
\foreach[evaluate=\y using int(\x-1)] \x in {5, 6, 9, 11}{
\draw [curvededge] (v\x) to (v\y);
\draw [curvededge] (v\y) to (v\x);
}
\foreach[evaluate=\y using int(\x+3)] \x in {2, 6, 8}
\draw [edge] (v\x) to (v\y);
\end{tikzpicture}
\end{minipage}
\begin{minipage}{0.24\textwidth}
\begin{tikzpicture}[auto]
\tikzstyle{vertex}=[draw, circle, inner sep=1.5mm, scale=0.5]
\tikzset{edge/.style = {->,> = Latex}}
\tikzset{curvededge/.style = {->,> = Latex, bend right}}
\node (v1) at (0,0) [vertex, label=$v$] {};
\node (v2) at (0,-1.1) [vertex] {};
\node (v3) at  (0,-2.2) [vertex] {};
\node (v4) at (1.1,0) [vertex] {};
\node (v5) at (1.1,-1.1) [vertex] {};
\node (v6) at  (1.1,-2.2) [vertex] {};
\node (v7) at (2.2,0) [vertex] {};
\node (v8) at (2.2,-1.1) [vertex,fill=blue] {};
\node (v9) at  (2.2,-2.2) [vertex, fill=blue] {};
\node (v10) at (3.3,0) [vertex] {};
\node (v11) at (3.3,-1.1) [vertex, fill=blue] {};
\node (v12) at  (3.3,-2.2) [vertex, fill=blue] {};
\foreach[evaluate=\y using int(\x+1)] \x in {1, 7, 11}
\draw [edge] (v\x) to (v\y);
\foreach[evaluate=\y using int(\x-1)] \x in {3}
\draw [edge] (v\x) to (v\y);
\foreach[evaluate=\y using int(\x+3)] \x in {1, 3, 4, 9}
\draw [edge] (v\x) to (v\y);
\foreach[evaluate=\y using int(\x-3)] \x in {8, 10}
\draw [edge] (v\x) to (v\y);
\foreach[evaluate=\y using int(\x-1)] \x in {5, 6, 9, 11}{
\draw [curvededge] (v\x) to (v\y);
\draw [curvededge] (v\y) to (v\x);
}
\foreach[evaluate=\y using int(\x+3)] \x in {2, 6, 8}
\draw [edge] (v\x) to (v\y);
\end{tikzpicture}
\end{minipage}
\caption{From left to right: a ZFS, a stalled FZFS, a critical set, and a FZFS that is not stalled.}
\label{fig:simpleexample}
\end{center}
\end{figure}

For any digraph $D$, we say that $G$ is the \emph{underlying graph} of $D$, denoted $G = UG(D)$, if $G$ is the unique simple, finite undirected graph obtained by replacing every arc $uv \in A(D)$ with an undirected edge $\{u, v\}$.  The digraph $D$ is \emph{weakly connected} if $UG(D)$ is connected.  We present several results related to paths and cycles in this paper.  

\begin{dfn}
A \emph{weak path} (resp. \emph{weak cycle}) is a digraph whose underlying graph is a path (resp. cycle). Given a digraph $D$ with an alternating sequence $P = v_1 a_1 v_2 a_2  v_3 a_3 \ldots v_{k-1} a_{k-1} v_k$  of  vertices $v_i \in V(D)$ and  distinct arcs $a_j \in A(D)$ such that the tail of $a_i$ is $v_i$ and the head of $a_i$ is $v_{i+1}$, if all vertices are distinct, then $P$ is a \emph{(directed) path}.  If vertices $v_1$ through $v_{k-1}$ are distinct, and $v_k = v_1$, then $P$ is a \emph{(directed) cycle}.   \end{dfn}

For both weak paths and weak cycles, we include the possibilities that $D=K_1$ (a single vertex) and that $UG(D) = K_2$. An example of a directed cycle is shown in Figure \ref{fig:orientedcycle}.

This remainder of this paper is organized as follows. We now describe motivation for the study of failed zero forcing.  In Section \ref{sec:extremevalues}, we characterize digraphs with high and low values of $\F(D)$.  In Section \ref{sec:comparing}, we provide a characterization of digraphs that have the unusual property that $\F(D) < \Z(D)$, which also results in a characterization of digraphs that have the property that $W \subseteq V$ is a critical set if and only if $|W| \geq k$ for some $k \geq 1$.   In Section \ref{sec:selectedgraphs} we determine $\F(D)$ for specific families of digraphs, including weak paths and cycles,  disconnected digraphs in terms of their components, some trees including oriented trees, and lastly, for weakly connected line digraphs such as de Bruijn and Kautz digraphs.  Finally, in Section \ref{sec:openproblems} we discuss possible directions for future research resulting from this paper.  Throughout the paper we focus on digraphs that do not have loops.  However, in Section \ref{sec:line} we consider digraphs that have loops to allow us to consider applicable line digraphs such as de Bruijn digraphs.  

\subsection{Zero forcing and minimum rank}

Zero forcing problems have been studied for their applications to minimum rank problems \cite{barioli2010zero, aim2008zero} as well as to identification and control in quantum networks \cite{burgarth2009local, burgarth2013zero,  burgarth2007full,  burgarth2009indirect, severini2008nondiscriminatory}.   The failed zero forcing number naturally relates to these applications as well.  We consider the connection of failed zero forcing to matrices here. 

Given a square matrix $M$ with $n$ rows, we use $ker(M)$ to denote the \emph{kernel} of $M$.  That is, for a vector $\mathbf{v}$ of length $n$, $\mathbf{v} \in ker(M)$ if and only if $M\mathbf{v}=0$.  The \emph{support} of a vector $\mathbf{x}=[x_i]$, denoted $\mbox{supp}(\mathbf{x})$, is given by $\{i : x_i \neq 0\}$.  Given a digraph $D$ with $|V|=n$, let $\mathcal{S}(D)$ denote the set of $n\times n$ matrices such that the entry in Row $i$, Column $j$ is nonzero if and only if $ij \in A$ for $i \neq j$, with diagonal entries unrestricted.

We note a proposition, similar to \cite[Proposition 2.3]{aim2008zero} but extended to digraphs.  The proof is similar to that of  \cite[Proposition 2.3]{aim2008zero} but included here for completeness.

\begin{prop}
Let $Z$ be a ZFS of digraph $D$  with $M \in \mathcal{S}(D)$.   If $\mathbf{x} \in ker(M)$ and $\mbox{supp}(\mathbf{x}) \cap Z = \emptyset$, then $\mathbf{x} = 0$.  
\end{prop}

\begin{proof}
If $Z=V$, then $\mbox{supp}(\mathbf{x})$ is empty, giving us $\mathbf{x}=0$, so suppose $Z \subsetneq V$.  Then $B^0(Z) \subsetneq B^1(Z)$, and there exist $u, v \in V$ with $\{v\} = N_D^+[u] \backslash B^0(Z)$. By assumption, $x_u = 0$ and $(Mx)_u = 0$.  The only nonzero entries in Row $u$ of $M$ are those corresponding to $N_D^+(u)$.  Since $\left( N_D^+(u) \backslash \{v\} \right) \subseteq Z$,    the corresponding entries of $\mathbf{x}$ are $0$, other than $x_v$.   We have the equation $M_{uv} x_v = 0$, giving $x_v=0$.  This is true for each color change.  Thus, $\mathbf{x}=0$.   
\end{proof}

Since $\F(D)$ is the maximum order of any FZFS, any set of order $\F(D) + 1$ or bigger is a ZFS.  Combining this fact with the above proposition gives us the following.

\begin{prop}
Let $S \subseteq V(D)$ with $|S| \geq \F(D) + 1$.  Let $M \in \mathcal{S}(D)$.  If $\mathbf{x} \in ker(M)$ and all entries of $\mathbf{x}$ corresponding to $S$ are $0$, then $\mathbf{x}=0$.  
\end{prop}

Thus,  if $M \in \mathcal{S}(D)$ for a digraph $D$,  and $\mathbf{x} \in ker(M)$ with $\mathbf{x} \neq 0$, then $\mathbf{x}$ has at least $n-F(D)$ nonzero entries.

\section{Extreme values}
\label{sec:extremevalues}
In this section, we establish the relationship between failed zero forcing and critical sets to characterize digraphs with high and low values of $\F(D)$.  

\begin{obs}
For any critical set $W$ in a digraph $D$, $V(D) \backslash W$ is a failed zero forcing set.  \label{obs:critical}
\end{obs}
\begin{prop}
Let $1 \leq k \leq n$.  Then $\F(D) = n-k$ if and only if the smallest cardinality of any critical set in $D$ is $k$.     \label{prop:charnminusk}
\end{prop}
\begin{proof}
If $D$ contains a critical set $W$ of cardinality $k$, then $V \backslash W$ is a FZFS.  Thus, $\F(D) \geq n-k$ where $k$ is the smallest cardinality of any critical set.  For the reverse direction, suppose $S$ is a largest FZFS in $D$.  Then for any $v \in S$, $v$ has either no out-neighbors or two out-neighbors in $V \backslash S$, since otherwise, either $S$ is a ZFS or there is a larger FZFS than $S$.  That is, $V \backslash S$ is a critical set.  Hence $\F(D) \leq n-k$.   
\end{proof}
For high values, we can describe these digraphs as follows.   
\begin{corollary}
$\F(D)=n-1$ if and only if $D$ has a source, and $\F(D) = n-2$ if and only if there exist $u,v \in V$ with $N_D^-(u) \backslash \{v\} = N_D^-(v) \backslash \{u\}$ and $\deg^-(w) > 0$ for all $w \in V$.  
$F(D) = n -3$ if and only if all of the following conditions are satisfied.
\begin{enumerate}
\item $\deg^-(v) > 0$ for all $v \in V$
\item For any  distinct vertices $u, v \in V$, $N_D^-(u) \backslash \{v\} \neq N_D^-(v) \backslash \{u\}$
\item There exist vertices $u, v, w \in V$ such that 
\begin{itemize}
\item $N_D^-(u) \backslash \{v, w\} \subseteq N_D^-(v) \cup N_D^-(w)$, 
\item $N_D^-(v) \backslash \{u, w\} \subseteq N_D^-(u) \cup N_D^-(w)$, and 
\item $N_D^-(w) \backslash \{u,v\} \subseteq N_D^-(u) \cup N_D^-(v)$.
\end{itemize}
\end{enumerate}
 \label{cor:source}
\end{corollary}

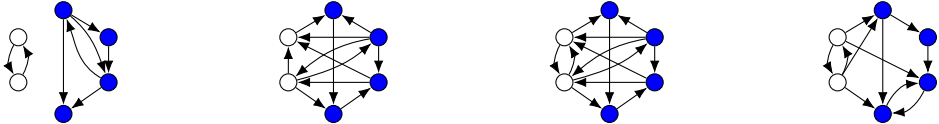
\begin{figure}[h!]
\begin{minipage}{0.22\textwidth}
\begin{center}
\begin{tikzpicture}[auto, scale=0.8]
\def\Radius{1cm}
\tikzstyle{vertex}=[draw, circle, inner sep=0.8mm]
\tikzset{edge/.style = {->,> = Latex}}
\node (u) at (0,0) [vertex]{};
\node (v) at (0,-0.75) [vertex]{};
\node (v1) at (0.75,0.45) [vertex, fill=blue]{};
\node (v2) at (1.5, 0) [vertex, fill=blue]{};
\node (v3) at (1.5,-0.75) [vertex, fill=blue]{};
\node (v4) at (0.75,-1.275) [vertex, fill=blue]{};
\draw[edge, bend right] (v) to (u);
\draw[edge, bend right] (u) to (v);
\draw[edge](v1) to (v4);
\draw[edge](v1) to (v2);
\draw[edge] (v2) to (v3);
\draw[edge] (v3) to (v4);
\draw[edge, out=150, in=290] (v3) to (v1);
\draw[edge, out=320, in=110] (v1) to (v3);
\end{tikzpicture}
\end{center}
\end{minipage}
\begin{minipage}{0.22\textwidth}
\begin{center}
\begin{tikzpicture}[auto, scale=0.8]
\def\Radius{1cm}
\tikzstyle{vertex}=[draw, circle, inner sep=0.8mm]
\tikzset{edge/.style = {->,> = Latex}}
\node (u) at (0,0) [vertex]{};
\node (v) at (0,-0.75) [vertex]{};
\node (v1) at (0.75,0.45) [vertex, fill=blue]{};
\node (v2) at (1.5, 0) [vertex, fill=blue]{};
\node (v3) at (1.5,-0.75) [vertex, fill=blue]{};
\node (v4) at (0.75,-1.2755) [vertex, fill=blue]{};
\draw[edge] (v) to (u);
\draw[edge](u) to (v1);
\draw[edge](v) to (v4);
\draw[edge](v2) to (u);
\draw[edge, out=190, in=40](v2) to (v);
\draw[edge, out=15, in=220](v) to (v2);
\draw[edge](v3) to (u);
\draw[edge](v3) to (v);
\draw[edge](v4) to (v3);
\draw[edge](v1) to (v4);
\draw[edge](v2) to (v1);
\draw[edge](v2) to (v3);
\end{tikzpicture}
\end{center}
\end{minipage}
\begin{minipage}{0.22\textwidth}
\begin{center}
\begin{tikzpicture}[auto, scale=0.8]
\def\Radius{1cm}
\tikzstyle{vertex}=[draw, circle, inner sep=0.8mm]
\tikzset{edge/.style = {->,> = Latex}}
\node (u) at (0,0) [vertex]{};
\node (v) at (0,-0.75) [vertex]{};
\node (v1) at (0.75,0.45) [vertex, fill=blue]{};
\node (v2) at (1.5, 0) [vertex, fill=blue]{};
\node (v3) at (1.5,-0.75) [vertex, fill=blue]{};
\node (v4) at (0.75,-1.2755) [vertex, fill=blue]{};

\draw[edge, bend right] (v) to (u);
\draw[edge, bend right] (u) to (v);
\draw[edge](u) to (v1);
\draw[edge](v) to (v4);

\draw[edge](v2) to (u);
\draw[edge, out=190, in=40](v2) to (v);
\draw[edge, out=15, in=220](v) to (v2);
\draw[edge](v3) to (u);
\draw[edge](v3) to (v);
\draw[edge](v4) to (v3);
\draw[edge](v1) to (v4);
\draw[edge](v2) to (v1);
\draw[edge](v2) to (v3);
\end{tikzpicture}
\end{center}
\end{minipage}
\begin{minipage}{0.22\textwidth}
\begin{center}
\begin{tikzpicture}[auto, scale=0.8]
\def\Radius{1cm}
\tikzstyle{vertex}=[draw, circle, inner sep=0.8mm]
\tikzset{edge/.style = {->,> = Latex}}
\node (u) at (0,0) [vertex]{};
\node (v) at (0,-0.75) [vertex]{};
\node (v1) at (0.75,0.45) [vertex, fill=blue]{};
\node (v2) at (1.5, 0) [vertex, fill=blue]{};
\node (v3) at (1.5,-0.75) [vertex, fill=blue]{};
\node (v4) at (0.75,-1.275) [vertex, fill=blue]{};
\draw[edge, bend right] (v) to (u);
\draw[edge, bend right] (u) to (v);
\draw[edge](u) to (v1);
\draw[edge](u) to (v3);
\draw[edge](v) to (v4);
\draw[edge](v) to (v1);
\draw[edge](v1) to (v4);
\draw[edge](v1) to (v2);
\draw[edge] (v2) to (v3);
\draw[edge, bend left] (v3) to (v4);
\draw[edge, bend left] (v4) to (v3);
\end{tikzpicture}
\end{center}
\end{minipage}
\caption{Examples of digraphs with $\F(D)=n-2$.}
\label{fig:nminus2}
\end{figure}

We now characterize digraphs that have  $\F(D)=0$.   Note that this is of particular interest because a digraph $D$ with $\F(D) = 0$ has the property that $\{v\}$ is a ZFS for any $v \in V(D)$.

\begin{theorem} 
$\F(D)=0$ if and only if $D$ is a directed cycle.\label{thm:f0}
\end{theorem}
\begin{proof}
Suppose $D$ is a directed cycle.  For any $W \subsetneq V(D)$, there exists at least one $w \in W$ with an in-neighbor $v \in V \backslash W$. Since $N^+(v) = \{w\}$, $W$ is not a critical set.  Thus, the smallest critical set is $W = V$, giving us $\F(D) = 0$ by Proposition \ref{prop:charnminusk}.  

For the other direction, suppose $\F(D)=0$.  Then $S = \{v\}$ is a ZFS for any $v \in V$.  If $|V|=1$, then we're done.  Otherwise,   $\deg^+(v)=1$ for any $v \in V$ to allow $B^0(S) \subsetneq B^1(S)$, and by Corollary \ref{cor:source}, $\deg^-(v) \geq 1$.  Since $\sum_{v \in V} \deg^+(v)  = \sum_{v \in V} \deg^-(v) =n$, $\deg^-(v)=1$ for all $v \in V$.  Noting that $D$ must be connected (else, all vertices in the largest connected component form a FZFS), we have that $D$ is a directed cycle.    \end{proof}

\begin{figure}[htbp]
\begin{center}
\begin{tikzpicture}[auto]
\def\Radius{0.75cm}
\tikzstyle{vertex}=[draw, circle, inner sep=0.8mm]
\tikzset{edge/.style = {->,> = Latex}}
\node (v0) at (0:\Radius) [vertex]{};
\foreach \y in {60, 120, ..., 300}{
\node (v\y) at (\y:\Radius) [vertex]{};}
\foreach \y [evaluate=\y as \x using {int(\y+60)}] in {0, 60, ..., 240}{
\draw[edge] (v\y) to (v\x);}
\draw[edge] (v300) to (v0);
\end{tikzpicture}
\end{center}
\caption{A digraph with $\F(D)=0$.  Every vertex is a ZFS.}
\label{fig:orientedcycle}
\end{figure}
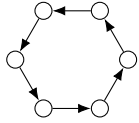 

An example of a digraph with $\F(D)=0$ is shown in Figure \ref{fig:orientedcycle}.  The following corollaries are immediate from Theorem \ref{thm:f0}.
\begin{corollary}
For every $v \in V$, $\{v\}$ is a ZFS of $D$ if and only if $D$ is a directed cycle.
\end{corollary}

\begin{corollary}
The digraph $D$ has no critical sets of cardinality less than $n$ if and only if $D$ is a directed cycle on $n$ vertices.
\end{corollary}

\section{Comparing $\F(D)$ with $\Z(D)$}
\label{sec:comparing}

In this section, we compare  $\F(D)$ with $\Z(D)$.  Specifically, we provide a characterization of digraphs  for which $\F(D) < \Z(D)$, which leads to a characterization of digraphs such that $W \subseteq V$ is a critical set if and only if $|W| \geq k$ for some integer $k$.    

\begin{obs}
The following are equivalent.
\begin{enumerate}
\item $\F(D) < \Z(D)$
\item  $\F(D)=\Z(D)-1$.  
\item $S \subseteq V(D)$ is a ZFS if and only if $|S| \geq \Z(D)$. 
\item $W \subseteq V(D)$ is a critical set if and only if $|W| \geq n - \F(D)$.  
\end{enumerate}
\end{obs}

\begin{lemma}
Suppose $D$ is a digraph with $\F(D) < \Z(D)$.  Then for each $v \in V$, $\deg^+(v) \geq \Z(D)$ or $v$ is a sink. \label{lem:minoutdegree}
\end{lemma}

\begin{proof}
Let $v \in V$.  Suppose $0<\deg^+(v)<\Z(D)$.  Then  $\deg^+(v) \leq\Z(D) -1 = \F(D)$.  Let $u \in N_D^+(v)$, and let $S =  \left( N_D^+[v]\backslash \{u\}\right) \cup S'$, where $S'$ is any $\F(D)-\deg^+(v)$ vertices in $V \backslash N_D^+[v]$.   Then $|S| = \F(D)$, so $S$ is a maximum FZFS.  
	But since $\{u\} = N_D^+[v] \backslash S$, it follows that $S$ is not stalled, a contradiction.  Hence $\deg^+(v) \geq \Z(D)$ or $\deg^+(v)=0$.  
 \end{proof}

Lemma \ref{lem:minoutdegree} leads to the following observation and lemma.

\begin{obs}
If $D$ is a digraph with  $\F(D) < \Z(D)$, then every set $S \subseteq V$ with $|S|=\Z(D)$ contains a vertex $v$ with $\deg^+(v) = \Z(D)$ and $vw \in A$ for all $w \in S \backslash \{v\}$.  \label{obs:degequalz}
\end{obs}

\begin{lemma}
If  digraph $D$ has  $\F(D) < \Z(D)$, then $\Z(D)\in \{1, 2, n-2, n-1, n\}$.  If $\Z(D)=2$, then $n \leq 5$.  \label{lem:valuesofn}
\end{lemma}

\begin{proof}
There are $n \choose \Z(D)$ sets $S$ with $|S|=\Z(D)$ in $D$. By Observation \ref{obs:degequalz}, there is a vertex $v$ in each set $S$ with $\deg^+(v) = \Z(D)$ and $vw \in A$ for all $w \in S \backslash \{v\}$. For any $v \in V$ with $\deg^+(v)=\Z(D)$, $v$ accounts for $\Z(D)$ sets $S$, giving us ${n \choose \Z(D)} \leq n\Z(D)$, which implies that $\Z(D) \leq 2$ or $\Z(D) \geq n-2$. 

If $\Z(D)=2$, then we have ${n \choose 2} \leq 2n$, which simplifies to $n \leq 5$.  
\end{proof}
We also make use of the following observation in our characterization of digraphs that have $\F(D) < \Z(D)$. 
\begin{obs}
Suppose $D$ is a digraph with $\deg^+(v) \leq 1$ for all $v \in V$.  For every $u, v \in V$, there exists $w \in V \backslash\{u, v\}$ such that exactly one of $wu \in A$ or $wv \in A$ is true if and only if $D$ consists of one of the following.
\begin{itemize}
\item  the union of vertex-disjoint cycles each of length at least $3$ that span all $n$ vertices, or
\item  the union of vertex-disjoint cycles each of length at least $3$ that span $n-1$ vertices, and 
\begin{itemize}  
\item a single isolated vertex, or
\item a single vertex that has exactly one other vertex as its out-neighbor.
\end{itemize}
\end{itemize} \label{obs:functionalgraph}
\end{obs}
Three examples of digraphs that satisfy Observation \ref{obs:functionalgraph} are shown in Figure \ref{fig:complement}.

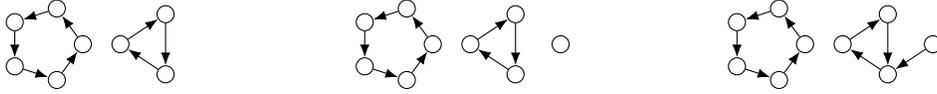
\begin{figure}[htbp]
\begin{center}
\begin{minipage}{0.3\textwidth}
\begin{center}
\begin{tikzpicture}
\def\Radius{0.5cm}
\tikzstyle{vertex}=[draw, circle, inner sep=0.8mm]
\tikzset{edge/.style = {->,> = Latex}}
\node (v0) at (0:\Radius) [vertex]{};
\foreach \y in {72, 144, ..., 288}{
\node (v\y) at (\y:\Radius) [vertex]{};}
\foreach \y [evaluate=\y as \x using {int(\y+72)}] in {0, 72, ..., 216}{
\draw[edge] (v\y) to (v\x);}
\draw[edge] (v288) to (v0);

\node(w0) at (1,0) [vertex]{};
\node(w1) at (1.6,0.4) [vertex]{};
\node(w2) at (1.6,-0.4) [vertex]{};

\foreach \y [evaluate=\y as \x using {int(\y+1)}] in {0, 1}{
\draw[edge] (w\y) to (w\x);}
\draw[edge] (w2) to (w0);
\end{tikzpicture}
\end{center}
\end{minipage}
\begin{minipage}{0.3\textwidth}
\begin{center}
\begin{tikzpicture}
\def\Radius{0.5cm}
\tikzstyle{vertex}=[draw, circle, inner sep=0.8mm]
\tikzset{edge/.style = {->,> = Latex}}
\node (v0) at (0:\Radius) [vertex]{};
\foreach \y in {72, 144, ..., 288}{
\node (v\y) at (\y:\Radius) [vertex]{};}
\foreach \y [evaluate=\y as \x using {int(\y+72)}] in {0, 72, ..., 216}{
\draw[edge] (v\y) to (v\x);}
\draw[edge] (v288) to (v0);

\node(w0) at (1,0) [vertex]{};
\node(w1) at (1.6,0.4) [vertex]{};
\node(w2) at (1.6,-0.4) [vertex]{};

\node(x) at (2.2, 0) [vertex]{};
\foreach \y [evaluate=\y as \x using {int(\y+1)}] in {0, 1}{
\draw[edge] (w\y) to (w\x);}
\draw[edge] (w2) to (w0);
\end{tikzpicture}
\end{center}
\end{minipage}
\begin{minipage}{0.3\textwidth}
\begin{center}
\begin{tikzpicture}
\def\Radius{0.5cm}
\tikzstyle{vertex}=[draw, circle, inner sep=0.8mm]
\tikzset{edge/.style = {->,> = Latex}}
\node (v0) at (0:\Radius) [vertex]{};
\foreach \y in {72, 144, ..., 288}{
\node (v\y) at (\y:\Radius) [vertex]{};}
\foreach \y [evaluate=\y as \x using {int(\y+72)}] in {0, 72, ..., 216}{
\draw[edge] (v\y) to (v\x);}
\draw[edge] (v288) to (v0);

\node(w0) at (1,0) [vertex]{};
\node(w1) at (1.6,0.4) [vertex]{};
\node(w2) at (1.6,-0.4) [vertex]{};

\node(x) at (2.2, 0) [vertex]{};
\foreach \y [evaluate=\y as \x using {int(\y+1)}] in {0, 1}{
\draw[edge] (w\y) to (w\x);}
\draw[edge] (w2) to (w0);
\draw[edge] (x) to (w2);
\end{tikzpicture}
\end{center}
\end{minipage}
\caption{Complements of some digraphs with $\F(D) < \Z(D)$}
\label{fig:complement}
\end{center}
\end{figure}

A \emph{tournament} $\og{K_n}$ is an oriented graph obtained by assigning an orientation to each edge in a complete graph, $K_n$.  A tournament is \emph{regular} if $\deg^+(v)=\deg^-(v)$ for every $v \in V$.
We also define a graph operation that we will use throughout the remainder of this section. 

\begin{dfn}
The  \emph{outjoin} from digraph $D$ to digraph $H$ denoted $D \overrightarrow{\vee} H$ is the digraph with vertex set $V(D \overrightarrow{\vee} H) = V(D) \cup V(H)$ and arc set $A(D \overrightarrow{\vee} H) = A(D) \cup A(H) \cup \{(v_{D}, v_{H}) : v_{D} \in V(D), v_{H} \in V(H)\}$.  
\end{dfn}

We now present the characterization of digraphs that have $\F(D) < \Z(D)$. Figure \ref{fig:complement} shows some examples of complements of digraphs that satisfy Item \ref{characterizationremovecycles} of Theorem \ref{thm:bigcharacterization}, and Figure \ref{fig:knoutjoinempty} shows an example of a digraph that satisfies Item \ref{characterizationcomplete}.

\begin{theorem}  \label{thm:bigcharacterization}
A digraph $D$ has $\F(D) < \Z(D)$ if and only if $D$ is one of the following.
\begin{enumerate}
\item a directed cycle. \label{characterizationdirectedcycle}
\item a regular tournament on 5 vertices. \label{tournament5}
\item A digraph obtained from $K_n$ by removing the arcs of  \label{characterizationremovecycles}
\begin{enumerate}
\item a collection of vertex-disjoint directed cycles each of length at least 3 that span $V$ ($n\geq 3$), \label{cyclesspan}
\item a collection of vertex-disjoint directed cycles each of length at least 3 that span $V \backslash \{v\}$ for some $v \in V$ ($n\geq 4$), or  \label{cyclesspanbutone}
\item $vu$ for some $u, v \in V$ and  a collection of vertex-disjoint directed cycles each of length at least 3 that span $V \backslash \{v\}$ ($n\geq 4$). \label{cyclesspanleaf}
\end{enumerate}
\item A digraph obtained from $K_{n-1}  \overrightarrow{\vee} \{v\}$ by removing the arcs of a collection of vertex-disjoint directed cycles  each of length at least 3 that span $K_{n-1}$ ($n \geq 4$).  \label{characterizationsinkcomplement}
\item  $K_j \overrightarrow{\vee} \overline{K_{\ell}}$ where $j \geq 2$ and $\ell \geq 0$. \label{characterizationcomplete}
\item $\overline{K_n}$. \label{characterizationisolated}
\end{enumerate}
\end{theorem}

\begin{proof} 
%%%%
%% FORWARD DIRECTION
%%%%
%Z=1
For the forward direction, suppose $\F(D)<\Z(D)$.  By Lemma \ref{lem:valuesofn}, $\Z(D)\in \{1, 2, n-2, n-1, n\}$.  

$\mathbf{\Z(D)=1}$:
If $\F(D) < \Z(D) = 1$, by Theorem \ref{thm:f0} we have that $D$ is a directed cycle,  Item \ref{characterizationdirectedcycle}.  

$\mathbf{\Z(D)=2}$:
Suppose $\Z(D)=2$ and $\F(D)=1$.  By Lemmas  \ref{lem:minoutdegree} and \ref{lem:valuesofn}, $3 \leq n \leq 5$.     If $n=3$, then $\deg^+(v)=2$ for each $v\in V$, implying $D=K_3$, Item \ref{characterizationcomplete}.  If $n=4$, then $\Z(D)=2 =n-2$, and is discussed with the case $\Z(D)=n-2$.  Suppose $n=5$. Note that for any $u, w \in V$, $uw \in A$ or $wu \in A$ since if not, $\{u, w\}$ is a FZFS of order $2$.  We show that $\deg^+(v)=2$ for each $v \in V$.  Suppose $D$ contains a vertex $v$ that is either a sink or $\deg^+(v) \geq 3$.  Any pair of vertices with this property would form a FZFS of order $2$, a contradiction, so $D$ contains at most one such vertex.  Note that if $|N^+(u) \backslash \{v\}| \geq 2$ for some $u \in V \backslash \{v\}$, then $\{u, v\}$ forms a FZFS of order $2$, a contradiction.  Hence $uv \in A$ and $\deg^+(u) = 2$ for each $u \in V \backslash \{v\}$.  But then $|V \backslash \{v\}| = 4$ with only $4$ arcs among $V \backslash \{v\}$, implying that $uw \notin A$ and $wu \notin A$ for some $u, w \in V \backslash \{v\}$, a contradiction.  Hence, if  $n=5$ and $\F(D)< \Z(D)=2$, then $\deg^+(v)=2$ for every $v \in V$, implying $\sum_{v \in V} \deg^+(v) = 10$.   Since we know that $uw \in A$ or $wu \in A$ for each $u, w \in V$, we have $|A| \geq {5 \choose 2} = 10$.  Hence, we must have that $D$ satisfies Item \ref{tournament5}.  

$\mathbf{\Z(D)=n-2}$:
First, assume that $D$ has no sink.  Since $\Z(D)=n-2$,  for every $u, v \in V$ there exists some $w \in V$ such that $wu \notin A$ or $wv \notin A$.  Assume without loss of generality $wu \notin A$.  Since $\deg^+(w) \geq n-2$ by Lemma \ref{lem:minoutdegree},  $wx \in A$ for every $x \in V \backslash \{u\}$, and $\overline{D}$ has the property that $\deg^+(x) \leq 1$ for each $x \in V$.  By applying Observation \ref{obs:functionalgraph} to $\overline{D}$, we conclude that if $D$ has no sink and $\F(D) <\Z(D)=n-2$, then $D$ satisfies one of Items \ref{cyclesspan} -- \ref{cyclesspanleaf}.  

If $D$ has a sink $v$, we  show that $uv \in A$ for each $u \in V \backslash \{v\}$.  Suppose $uv \notin A$.  Then $V \backslash \{u, v\}$ is a ZFS, so there exists $w \in V \backslash \{u, v\}$ such that exactly one of $wv \notin A$  or $wu \notin A$ is true.  We first consider the case that $wu \in A$ for all $w \in V \backslash \{u, v\}$.  Thus, $w_0v \notin A$ for some $w_0 \in V \backslash \{u, v\}$.  Now, $V \backslash \{w_0, u\}$ is a ZFS, and we just assumed that $wu \in A$  for all $w \in V \backslash \{u, v\}$, so there exists $w_1 \in V \backslash \{u, v, w_0\}$ such that $w_1 w_0 \notin A$.  Note that $w_1 \notin  \{u, v, w_0\}$ since $v$ is a sink, and by Lemma \ref{lem:minoutdegree}, $\deg^+(u), \deg^+(w_0) \geq n-2$.  Since $V \backslash \{u, w_1\}$ is a ZFS, and we assumed $wu \in A$ for all $w \in V$, there exists $w_2 \in V \backslash \{u, v, w_0, w_1\}$ such that $w_2 w_1 \notin A$.  We can repeat this argument indefinitely, but since $|V|$ is finite, eventually run out of vertices.  

Thus it must be that $w_0 u \notin A$ (but $w_0 v \in A$) for some $w_0 \in V$.  If there is some vertex $w_1 \in V \backslash\{w_0, u, v\}$ such that $w_1 w_0 \notin A$, but $w_1 u \in A$, then consider whether there is a vertex $w_2 \in V \backslash\{w_0, w_1 u, v\}$ such that $w_2 w_1 \notin A$ but $w_2 w_0 \in A$, and so forth until we come to $w_i$ that has no such vertex $w_{i+1}$.   Then $V \backslash \{w_i, w_{i-1}\}$ (or $V \backslash \{w_i, u\}$ if $i=0$) is a ZFS, so there must exist $x_0 \in V$ such that $x_0 w_{i-1} \notin A$ (or $x _0u \notin A$ if $i=0$) but $x_0 w_i \in A$.  We perform the same argument on $x_0$, $x_1$ etc. as on $w_0$, $w_1$ etc., until we find $x_j$ for which there is no $x_{j+1}$.  Then $V \backslash \{x_j, w_i\}$ forms a FZFS, a contradiction.  Hence, if $v$ is a sink, then $uv \in A$ for each $u \in V \backslash \{v\}$, also implying that $D$ has at most one sink.  

Suppose $D$ has a sink $v$.  Since $\{u, v\}$ is a ZFS and $uv \in A$ for every $u \in V \backslash \{v\}$, every $u \in V \backslash \{v\}$ has the property that $wu \notin V$ for some $w \in V \backslash \{u, v\}$.  Recalling that $\deg^+(u) \geq n-2$ for each $u \in V \backslash \{v\}$, the complement of the digraph induced by $V \backslash \{v\}$ is the union of vertex-disjoint directed cycles each of length at least $3$.  Consequently, $D$ satisfies Item \ref{characterizationsinkcomplement}.  

$\mathbf{\Z(D)=n-1}$:
By Lemma \ref{lem:minoutdegree}, if $\F(D) <\Z(D)=n-1$, then for any $v \in V$ with $\deg^+(v) >0$,  we have  
$\deg^+(v) =n-1$.  If there are no vertices with $\deg^+(v)>0$, then $D$ consists of a set of isolated vertices, which has $\Z(D) =n$, a contradiction.  If there is exactly one vertex $v$ with  $\deg^+(v)>0$, then $V \backslash \{v\}$ is a FZFS of order $n-1$, a contradiction.  Hence, there are at least $2$ vertices with out-degree $n-1$, and $D$ satisfies Item \ref{characterizationcomplete}.  

$\mathbf{\Z(D)=n}$:  If $\Z(D)=n$, then $A = \emptyset$.  Hence $D$ satisfies Item \ref{characterizationisolated}.
%%%%
%% REVERSE DIRECTION
%%%%

For the reverse direction, Item \ref{characterizationdirectedcycle} was established in Theorem \ref{thm:f0}.  

For Item  \ref{tournament5}, since $\deg^+(v)=2$ for every $v \in V$, every vertex is a FZFS.  Let $S=\{u,v\}$ for any $u, v \in V$.  Either $uv \in A$ or $vu \in A$.  Assume without loss of generality $uv \in A$.  Then $N^+[u] \backslash S =\{w\}$ for some $w \in V$, so $B^1(S) =\{u, v, w\}$.  Either $wv \in A$ or $vw \in A$, so there exists $x \in V$ with $\{x\} = N^+[w] \backslash B^1(S)$ (without loss of generality).   Finally, $B^2(S) = V \backslash \{y\}$ for some $y \in V$, and $\deg^-(y) = 2$, so $B^3(S)=V$, and $S$ is a ZFS.   Hence, $\F(D) < \Z(D)=2$.  

Suppose $D$ satisfies Item \ref{characterizationremovecycles} or Item \ref{characterizationsinkcomplement}.  Pick any $u, v \in V$ and let $S = V \backslash \{u,v\}$.  Then there exists $w \in S$ such that without loss of generality $wu \notin A$ and $wv \in A$, so $B^1(S) = V \backslash \{u\}$.  The vertex $u$ has $\deg^-(u) \geq n-2$, giving us $B^2(S)=V$.  Thus any set $S \subseteq V$ with $|S|=n-2$ is a ZFS.  Let $X = V \backslash \{u, v, w\}$ for any $u, v, w \in V$.  Let $x \in X$.  Then either $\deg^+(x)=0$ if $x$ is a sink, in which case $xu, xv, xw \notin A$, or $\deg^+(x) \geq n-2$, implying that at most one of $\{xu, xv, xw\}$ is not an arc in $D$.  Hence any set $X$ with $|X|=n-3$ is a FZFS, giving us that $\F(D) < \Z(D) = n-2$.

Suppose $D$ satisfies Item \ref{characterizationcomplete}.  Let $S = V \backslash \{v\}$ for any $v \in V$.  Since $\deg^-(v) \geq 1$, $B^1(S)=V$, and $S$ is a ZFS with $|S|=n-1$.  If $n=j=2$ and $\ell =0$, we have $\F(D)=0$ and we are done.  Otherwise,  let $X = V \backslash \{u, v\}$ for any $u, v \in V$.  Then $N^-(u)\cap X =N^-(v) \cap X$; hence $B^1(X)=X$, and $X$ is a FZFS with $|X|=n-2$.  Hence, $\F(K_j \overrightarrow{\vee} \overline{K_{\ell}} )  < \Z( K_j \overrightarrow{\vee} \overline{K_{\ell}} )$. 

Finally, Item \ref{characterizationisolated} was established in \cite{fetcie2014failed}, completing the characterization.
\end{proof}

\begin{figure}[h!]
\begin{center}
\begin{tikzpicture}[auto]
\tikzstyle{vertex}=[draw, circle, inner sep=0.8mm]
\tikzset{edge/.style = {->,> = Latex}}
\def\Radius{0.9cm}
\def\number{5}

\foreach \y  in {1, ..., \number}{
\node (v\y) at ({360/\number*(\y-1)}:\Radius) [vertex]{};}
\node (vin) at (2.2*\Radius, 0.6*\Radius) [vertex]{};
\node (vin2) at (2.2*\Radius, -0.6*\Radius) [vertex]{};
\foreach \y  in {2, ..., \number}{
\draw [edge, bend right] (v\y) to (v1);}

\foreach \y  in {1, 3, 4, ..., \number}{
\draw [edge, bend right] (v\y) to  (v2);}

\foreach \y  in {1, 2, 4, 5}{
\draw [edge, bend right] (v\y) to(v3);}
\foreach \y  in {1, 2, 3, 5}{
\draw [edge, bend right] (v\y) to(v4);}
\foreach \y  in {1, 2, 3, 4}{
\draw [edge, bend right] (v\y) to(v5);}

\foreach \y in {1,2,...,\number}{
\draw[edge](v\y) to (vin);
\draw[edge](v\y) to (vin2);
}

\end{tikzpicture}
\end{center}
\caption{ $\F(K_5 \protect\overrightarrow{\vee}  \overline{K_2})=5$ and  $\Z(K_5 \protect\overrightarrow{\vee}  \overline{K_2}) =6$.}
\label{fig:knoutjoinempty}
\end{figure}
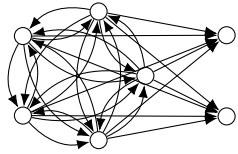

Restricting Theorem \ref{thm:bigcharacterization} to oriented graphs gives us the following characterization.
\begin{corollary}
An oriented graph $\ogg$ has the property that $\F(\ogg) < \Z(\ogg)$ if and only if $\ogg$ is one of the following.

\begin{enumerate}
\item $\overline{K}_n$, \label{emptygraph}
\item a directed cycle, \label{directedcycle}
\item a directed 3-cycle with all 3 vertices outjoined to one additional vertex,  \label{tournamentextravertex}
\item a regular tournament on 5 vertices. \label{tournament}
\end{enumerate}
\end{corollary}
 
\begin{center}
\begin{figure}[h!]
\begin{minipage}{1.5in}
\begin{center}
\begin{tikzpicture}[auto, scale=0.8]
\def\Radius{0.8cm}
\tikzstyle{vertex}=[draw, circle, inner sep=0.8mm]
\tikzstyle{dottedvertex}=[draw, circle, dotted, inner sep=0.8mm]
\tikzset{edge/.style = {->,> = Latex}}
\node (v1) at (0, 0)[vertex]{};
\node (v2) at (0.6, 0.3)[vertex]{};
\node (v3) at (0.9, -0.3)[vertex]{};
\node (v4) at (0.3, -0.6)[vertex]{};
\end{tikzpicture}
\end{center}
\end{minipage}
\begin{minipage}{1.5in}
\begin{center}
\begin{tikzpicture}[auto, scale=0.8]
\def\Radius{0.7cm}
\tikzstyle{vertex}=[draw, circle, inner sep=0.8mm]
\tikzstyle{dottedvertex}=[draw, circle, dotted, inner sep=0.8mm]
\tikzset{edge/.style = {->,> = Latex}}
\foreach \y in {0, 60}{
\node (v\y) at (\y:\Radius) [vertex]{};}
\foreach \y in {120}{
\node (v\y) at (\y:\Radius) [vertex]{};}
\foreach \y [evaluate=\y as \x using {int(\y+60)}] in {0, 60}{
\draw[edge] (v\y) to (v\x);}
\draw[edge, dotted] (v120) to (135:\Radius) to (150:\Radius) to (165:\Radius) to (180:\Radius) to (195:\Radius) to (210:\Radius) to (225:\Radius) to (240:\Radius) to (255:\Radius) to (270:\Radius) to (285:\Radius) to (300:\Radius) to (315:\Radius) to (330:\Radius) to (345:\Radius) to (v0);
\end{tikzpicture}
\end{center}
\end{minipage}
\begin{minipage}{1.5in}
\begin{center}
\begin{tikzpicture}[auto, scale=0.8]
\def\Radius{0.7cm}
\tikzstyle{vertex}=[draw, circle, inner sep=0.8mm]
\tikzstyle{dottedvertex}=[draw, circle, dotted, inner sep=0.8mm]
\tikzset{edge/.style = {->,> = Latex}}
\foreach \y in {0, 1,2}{
\node (v\y) at (120*\y:\Radius) [vertex]{};}
\node (vin) at (1.5, 0) [vertex]{};
\draw[edge] (v0) to (v1);
\draw[edge] (v1) to (v2);
\draw[edge] (v2) to (v0);
\draw[edge] (v0) to (vin);
\draw[edge, bend left] (v1) to (vin);
\draw[edge, bend right] (v2) to (vin);
\end{tikzpicture}
\end{center}
\end{minipage}
\begin{minipage}{1.5in}
\begin{center}
\begin{tikzpicture}[auto, scale=0.8]
\def\Radius{0.7cm}
\tikzstyle{vertex}=[draw, circle, inner sep=0.8mm]
\tikzstyle{dottedvertex}=[draw, circle, dotted, inner sep=0.8mm]
\tikzset{edge/.style = {->,> = Latex}}
\foreach \y in {0, 1, ..., 4}{
\node (v\y) at (72*\y:\Radius) [vertex]{};}
\draw[edge] (v0) to (v1);
\draw[edge] (v0) to (v2);

\draw[edge] (v1) to (v2);
\draw[edge] (v1) to (v3);

\draw[edge] (v2) to (v3);
\draw[edge] (v2) to (v4);

\draw[edge] (v3) to (v0);
\draw[edge] (v3) to (v4);

\draw[edge] (v4) to (v0);
\draw[edge] (v4) to (v1);
\end{tikzpicture}
\end{center}
\end{minipage}

\caption{All oriented graphs with $\F(\protect\overrightarrow{G}) < \Z(\protect\overrightarrow{G})$.  For the first two digraphs, $|V|\geq 1$.}
\label{fig:flessthanz}
\end{figure}
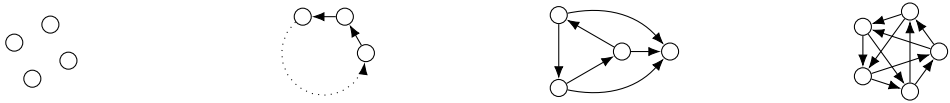 
\end{center}

\begin{corollary}
The following list contains the only digraphs with the property that there exists $k$ such that $W \subseteq V$ is a critical set if and only if $|W| \geq k$.   
\label{cor:criticalk}
\begin{enumerate}
\item $\overline{K}_n$,  ($k=1$), \label{emptygraphcritical}
\item  $K_j \overrightarrow{\vee} \overline{K_{\ell}}$ where $j \geq 2$ and $\ell \geq 0$, ($k=2$),
\item a digraph obtained from $K_n$ (where $n\geq 4$ except for (a) which allows $n \geq 3$) by removing the arcs of: (a) a collection of vertex-disjoint directed cycles each of length at least 3 that span $V$, (b) a collection of vertex-disjoint directed cycles each of length at least 3 that span $V \backslash \{v\}$ for some $v \in V$, or  (c) $vu$ for some $u, v \in V$ and  a collection of vertex-disjoint directed cycles each of length at least 3 that span $V \backslash \{v\}$, ($k=3$),
\item a digraph obtained from $K_{n-1}  \overrightarrow{\vee} \{v\}$ where $n \geq 4$ by removing the arcs of a collection of vertex-disjoint directed cycles  each of length at least 3 that span $K_{n-1}$, ($k=3$),\label{tournamentextravertexcritical}
\item a regular, non-transitive tournament on 5 vertices, ($k=4$), or \label{tournamentcritical}
\item a directed cycle, ($k=n$). \label{directedcyclecritical}
\end{enumerate}
\end{corollary}

\section{Select digraphs}
\label{sec:selectedgraphs}
For a digraph consisting of two or more components, we can determine the failed zero forcing number in terms of the failed zero forcing numbers and orders of the components.   The result is similar to the result for undirected graphs in  \cite{fetcie2014failed}.  

\begin{theorem}
Let $D$ be a digraph that consists of $k$ components where $k \geq 1$, and let $D_i$ denote the $i^{th}$ component of $D$, $1 \leq i \leq k$.  Then 
$$\F(D)=\max_{1 \leq j \leq k}\left(  \F(D_j) +  \sum_{i=1, i\neq j}^{k} |V(D_i)| \right).$$
\end{theorem}

\begin{proof}
If $k=1$, the result is trivial. Otherwise, for any FZFS $S_i$ of any component $D_i$, $S=\left(V(D) \backslash V(D_i)\right) \cup S_i$  is a FZFS.

 If $S' \subseteq V$ with $|S'| >\max_{1 \leq j \leq k}\left(  \F(D_j) +  \sum_{i=1, i\neq j}^{k} |V(D_i)| \right)$, then for any $\ell$, $|S' \cap V(D_{\ell})| > F(D_{\ell})$, so $S' \cap V(D_{\ell})$ is a ZFS of $D_{\ell}$, and consequently $S'$ is a ZFS of $D$.  
\end{proof}

A \emph{directed acyclic graph} is a digraph that contains no directed cycles.  The following proposition follows directly from Corollary \ref{cor:source}, since every directed acyclic graph has a source. 
\begin{prop}
For any directed acyclic graph $D$, $\F(D) = n-1$. \label{prop:dag}
\end{prop}

We turn our attention to special cases of directed trees, starting with oriented trees.   For any vertex $v$ in a directed tree, if $\left| N^+(v) \cup N^-(v) \right| =1$, then we say that $v$ is a \emph{leaf}.  The following corollary follows immediately from Proposition \ref{prop:dag}, since every oriented tree is a directed acyclic graph.

\begin{corollary}
For any oriented tree $\og{T}$,  $\F(\overrightarrow{T})= n-1$. \label{cor:orientedtree}
\end{corollary}

The (undirected) graph $K_{1,t}$ has a single vertex adjacent to $t=n-1$ other vertices, and no other edges.  

\begin{theorem}
If $UG(D) = K_{1,t}$ for any $t \geq 1$, then
$$\F(D)=\begin{cases}
t,  & \text{ if } D \text{ is oriented, or if at least one leaf has in-degree 0.}  \\
t-1,  & \text{ otherwise. }\\
\end{cases}
$$
\end{theorem}
\begin{proof}
If $D$ is oriented, then $\F(D)=t$ by Corollary \ref{cor:orientedtree}.  If $\deg^-(v)=0$ for a leaf $v$, then $v$ is a source, and by Corollary \ref{cor:source}, $\F(D)=t$.

Otherwise, there exist $u, w \in V(D)$ such that $uw, wu \in A(D)$, and $\deg^-(v)=1$ for every leaf $v \in V(D)$.  Thus, $D$ has no source, giving us that $\F(D) \leq t-1$.  If $UG(D)=K_2$, $D$ is a 2-cycle, and $\F(D)=0$, so we are done.  Let $u$ be the non-leaf vertex in $V$, and let $v, w \in V \backslash \{u\}$.  Then $\{v, w\}$ is a critical set because $u$ is the unique in-neighbor of both.  By Proposition \ref{prop:charnminusk}, $\F(D) = t-1$.  
\end{proof}

\subsection{Weak paths}
 
To establish $\F(D)$ if $UG(D) = P_n$, we assume that the vertices of $D$ are labeled in order from one end-vertex to the other: $v_1, v_2, \ldots, v_n$.

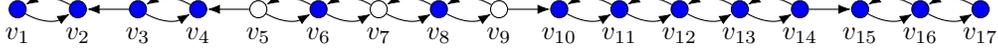
\begin{figure}[h!]
\begin{center}
\begin{tikzpicture}[auto]
\tikzstyle{vertex}=[draw, circle, inner sep=0.8mm]
\tikzset{edge/.style = {->,> = Latex}}
\tikzset{curvededge/.style = {->,> = Latex, bend right}}
\foreach \x in {1, 2, 3, 4, 6, 8, 10, 11, ..., 17}{
\node (v\x) at (0.8*\x,0)[vertex, fill=blue, label= below:$v_{\x}$]{};}
\foreach \x in {5, 7,9}{
\node (v\x) at (0.8*\x,0)[vertex,  label= below:$v_{\x}$]{};}
\draw[edge](v3)to(v2);
\draw[edge](v5)to(v4);
\draw[edge](v9)to(v10);
\draw[edge](v14)to(v15);
\foreach \x  [evaluate=\x as \y using {int(\x+1)}]  in {1, 3, 5, 6, ...,8, 10, 11, 12, 13, 15, 16}{
\draw[curvededge](v\x)to(v\y);
\draw[curvededge](v\y)to(v\x);}
\end{tikzpicture}
\caption{A weak path with maximum FZFS shown in blue.} \label{fig:path}
\end{center}
\end{figure}

\begin{theorem}
Suppose $D$ is a weak path.  Let $V_{-} =\{v_1\} \cup \{ v_k : v_kv_{k-1} \in A \mbox{ and } v_{k-1} v_k \notin A\}$, and  $V_{+} =\{v_n\} \cup \{ v_k : v_k v_{k+1} \in A \mbox{ and } v_{k+1} v_k \notin A\}$.    Let $\ell = \min\{i -j : v_j \in V_-, v_i \in V_+, i-j \geq 0\}$.  Then 
$$\F(D)= n - 1 - \left\lceil\frac{\ell}{2}\right\rceil.$$ \label{thm:path}
\end{theorem}

\begin{proof}
Let $S$ be defined as follows, where $i = i^*$ and $j=j^*$ are the indices achieving $\min\{i -j : v_j \in V_-, v_i \in V_+, i-j \geq 0\}$.
$$ S = 
\begin{cases}
V \backslash \{v_{j^*}, v_{j^*+2}, v_{j^*+4}, \ldots, v_{j^*+2k}\} & \mbox{ if } \ell= 2k \\
V \backslash \{v_{j^*}, v_{j^*+2}, v_{j^*+4}, \ldots, v_{j^*+2k}\} \cup \{v_{i^*}\} & \mbox{ if } \ell = 2k+1. \\
\end{cases}$$
 An example with $n=17$, $i^*=9$ and $j^*=5$ is shown in Figure \ref{fig:path}.

To show that $S$ is a FZFS, let $v_{s} \in S$.   If $s>i^*$ or $s<j^*$, then $N^+(v_{s})\subseteq S$.  If $j^* < s < i^*$, then $N^+(v_{s}) = \{ v_{s-1}, v_{s+1} \}$ (otherwise the minimum assumption is violated).  By construction of $S$, $v_{s-1}, v_{s+1} \notin S$.  Thus $B^1(S) = S$ and $S$ is a FZFS.

We show that $S$ is a maximum FZFS.  If $|S| = n-1$, then we are done.  Note that this includes any case with a source, so we can assume that $\deg^-(v) \geq 1$ for all $v \in V$.    Suppose there exists $S' \subseteq V$ with $|S| < |S'| < n$, and $S'$ is stalled.  For any closest pair of $u \in V_+$ and $w \in V_-$, if a pair of adjacent vertices between them is in $S'$, then all vertices from $u$ to $w$ are in $S'$, since $S'$ is stalled.  If this is true for all such pairs  $u$ and $w$ then $S'=V$, so  there must exist a pair $v_{s} \in V_+$ and $v_{t} \in V_-$ for which this is not the case.  Note since $|S'|>|S|$, there is at most one such pair and that $v_{s}, v_{t} \in S'$.  Also since $|S'|>|S|$, if $s>1$ then $v_{s-1} \in S'$, and if $t<n$ then $v_{t+1} \in S'$.  That is, $N^+[v_s] \backslash S' =\{v_{s+1}\}$, which is a contradiction: either $S'$ is not stalled, or two adjacent vertices between $v_s$ and $v_{t}$ are in $S'$.  Thus, $S$ is a maximum FZFS.   
\end{proof}

In many cases, the formula from Theorem \ref{thm:path} can be simplified.  If a weak path $D$ contains a source including if $D$ is an oriented path, for example, then there exists at least one vertex in $V_+\cap V_-$, giving us $i^*=j^*$, and consequently $\F(D) = n-1$.   By setting  $i^*=n$ and $j^*=1$ in Theorem \ref{thm:path}, we have the following corollary, established for undirected graphs in \cite{fetcie2014failed}.
\begin{corollary}
For any undirected path $P_n$ (or a weak path with $\{v_i v_{i+1}, v_{i+1} v_i\} \subseteq A$ for $1 \leq i \leq n-1$), $\F(P_n) = \left\lceil  \frac{n-2}{2}\right\rceil$.  

\end{corollary}

\subsection{Weak cycles}

We now turn to weak cycles, starting with oriented cycles.

\begin{obs}
Any oriented cycle $\og{C_n}$ that is not a directed cycle has a source. 
\end{obs}

As we know from Corollary \ref{cor:source}, 
if $D$ has a source then $\F(D)=n-1$.  Combining with Theorem \ref{thm:f0} completes the proof of the following theorem.  
\begin{theorem}
An oriented cycle $\overrightarrow{C_n}$  has
$$\F(\overrightarrow{C_n}) =
\begin{cases}
0, \mbox{ if it is a directed cycle } \\
n-1, \mbox{ otherwise } \label{thm:orientedcycle}
\end{cases}$$
\end{theorem}

Finally, we turn to weak cycles in general.  We present the failed zero forcing numbers of weak cycles depending on the orientations of the edges.  Given a weak cycle $D$, pick any vertex and label the vertices in order around the cycle, so $V=\{v_0, v_1, \ldots v_{n-1}\}$.   Let 
$$V_{-} = \{ i : v_{i+1} v_{i} \in A \mbox{ and } v_{i} v_{i+1} \notin A\},$$
$$V_{+} =\{ j : v_j v_{j+1} \in A \mbox{ and } v_{j+1} v_j \notin A\}, \mbox{and}$$ 
$$V_{0} = \{ k : v_k v_{k+1} \in A \mbox{ and } v_{k} v_{k+1} \in A\}$$
where we assume addition is modulo $n$, so for example if $v_i = v_{n-1}$, then $v_{i+1}=v_0$.

We define a \emph{run on} $k$  \emph{vertices} to be a consecutive sequence of vertices along the cycle all from the same set: $V_-, V_+$, or $V_0$.  We say that the run is \emph{maximal} if no vertex can be added to the run without violating the definition.

\begin{theorem}
Suppose $V_- $ or $V_+= \emptyset$ in a weak cycle $D$.    Let $\ell$  be the number of maximal runs of vertices in $V_0$.  Let $n_i$ denote the number of vertices in the $i^{th}$ maximal run of  vertices from  $V_0$, in order around the cycle.  Then 

$$\F(D) = \sum_{i=1}^{\ell} \left\lceil \frac{n_i}{2} \right\rceil$$ \label{thm:cycleonedirection}
\end{theorem}

\begin{proof}
Suppose $V_- = \emptyset$ (without loss of generality).  If $V_0 = \emptyset$ then $D$ is a directed cycle, so $\F(D)=0$ by Theorem \ref{thm:orientedcycle}.  Assume that $V_0 \neq \emptyset$.  

We define $S \subseteq V$ as follows.  The first run has $n_1$ vertices.  We can assume that the vertex labels begin with the first run, so the first run vertices are $v_0, v_1, \ldots, v_{n_1-1}$.  Add $v_1, v_3, \ldots$ up to  $v_{n_1-1}$  or $v_{n_1}$ (whichever is odd) to $S$.  This gives us $\lceil \frac{n_1}{2} \rceil$ vertices.  We do this for each maximal run of vertices from $V_0$, giving us $|S| =    \sum_{i=1}^{\ell} \left\lceil \frac{n_i}{2} \right\rceil.$

We show that $S$ is a maximum FZFS.  Since $\deg^+(v) =2$ for every $v \in S$ by construction, and since $N^+(S) \subseteq V \backslash S$, $S$ is stalled and therefore a FZFS. Now, suppose $S' \subseteq V$ with $|S'| > |S|$, and $S'$ is stalled.  Then either there must be some $v \in V$ that has $\deg^+(v) = 1$ with $v \in S'$, or there exists an $i$th run with more than $\lceil\frac{n_i}{2}\rceil$ vertices in $S'$.  In the first case, let $\{u\} = N^+(v)$, Since $S'$ is stalled, $u \in S'$.  However, recalling that $V_- = \emptyset$, since $S'$ is stalled $w \in S'$, where $\{w\} = N^+(u)$.  We can continue the same argument for each vertex along the cycle, giving us $S'=V$.  In the second case, suppose the $i$th run has more than $\lceil\frac{n_i}{2}\rceil$ vertices in $S'$.  Then either the first vertex in the run is in $S'$, or there are two adjacent vertices in the run that are in $S'$.  If there are two or more adjacent vertices  in the run that are in $S'$, let $x$ be the last such vertex.  Then $|N^+(x)\backslash S'|=1$, a contradiction since $S'$ is stalled.  Otherwise, let $x \in S'$ be the first vertex in the run.  We assumed that no adjacent vertices are in $S'$, so the next vertex in the run, $y$, is not in $S'$.  But $\{y\}=N^+(x)$, contradicting our assumption that $S'$ is stalled.  
Hence, $S$ is a maximum FZFS, and $\F(D) = \sum_{i=1}^{\ell} \left\lceil \frac{n_i}{2} \right\rceil.$
\end{proof}

An example of Theorem \ref{thm:cycleonedirection} with $V_- = \emptyset$ is shown in Figure \ref{fig:kunderhalfn}.  Theorem \ref{thm:cyclebothdirections} establishes $\F(D)$ in the case that $D$ has $V_0$, $V_+$, and $V_-$ nonempty.  Figure \ref{fig:koverhalfn} shows an example of Theorem \ref{thm:cyclebothdirections}.

\begin{theorem}
Let $D$ be a weak cycle such that $V_0$, $V_-$, and $V_+$ are nonempty.  Let $d(i,j) = j-i  \mod n$ and set $\ell = \min \{ d(i,j) : i \in V^-, j \in V^+\}$.  Then 
$$\F(D) = n - 1 - \left\lfloor \frac{\ell}{2}\right\rfloor.$$ \label{thm:cyclebothdirections}
\end{theorem}

\begin{proof} 
Let $(\hat{i}, \hat{j})$ be the indices that achieve $\ell = \min \{ d(i,j) : i \in V^-, j \in V^+\}$.  Define 

$$S = 
\begin{cases}
V \backslash \left\{ v_{\hat{i}+1}, v_{\hat{i}+3}, v_{\hat{i}+5}, \ldots, v_{\hat{i}+\ell} = v_{\hat{j}}  \right\} & \mbox{ if } \ell \mbox{ is odd} \\
V \backslash \{ v_{\hat{i}+1}, v_{\hat{i}+3}, v_{\hat{i}+5}, \ldots v_{\hat{i}+\ell-1} \} \cup \{ v_{\hat{j}} = v_{\hat{i}+\ell}  \} & \mbox{ if } \ell \mbox{ is even } \\
\end{cases}
$$
where all indices are taken modulo $n$.  We show that $V \backslash S$ is a critical set. If $v \in V \backslash S$, then $v = v_{\hat{i}+2m}$ for some nonnegative $m$, so if $u \in N^-(v)$ then $u = v_{\hat{i}+2m + 1}$ or $u = v_{\hat{i}+2m - 1}$.  If $u = v_{\hat{i}+2m+ 1}$, then $N^+(u) = \{  v, v_{\hat{i}+2m+2} \}$.  If $u = v_{\hat{i}+2m - 1}$, then $N^+(u) = \{  v, v_{\hat{i}+2m-2} \}$.  Thus, $V \backslash S$ is a critical set, and by Observation \ref{obs:critical}, $S$ is a FZFS.

Let $W$ be a critical set in $D$.  We show that $|W| \geq \lfloor \ell / 2 \rfloor + 1$.     
Choose any $i \in V_-$ and $j \in V_+$ such that $d(i,j)$ is minimal.  That is, if there exists $j^*$ with $d(i,j^*) < d(i,j)$ or $i^*$ such that $d(i^*,j)<d(i,j)$, then replace $j$ with $j^*$ or $i$ with $i^*$ as appropriate (or if both cases are true, pick one).  Do this until there exist no such $i^*$ or $j^*$.  

Let $P$ denote the weak path $v_i, v_{i+1}, \ldots, v_j, v_{j+1}$.  Note that for any $v_s \in P$, if $s \notin \{i, j, j+1\}$, then $s \in V_0$.  Let $P_1$ denote the weak path starting from $v_i$ and descending modulo $n$ (i.e, the weak path that is edge-disjoint from $P$) until the first vertex $v_x$ such that $x-1 \in V_+$.  Note that $v_x$ exists, because if no other vertex before satisfies the property, then $v_{j+1}$ is such a vertex.  Similarly, let $P_2$ be the weak path starting from $v_{j+1}$ and ascending modulo $n$ (i.e, the weak path that is edge-disjoint from $P$) until the first vertex $v_y$ with $y \in V_-$.  

Note that if there exist adjacent vertices in $V(P) \cap V \backslash W$, then $\left( V(P) \cup V(P_1) \cup   V(P_2) \right)  \cap W = \emptyset$, because otherwise there exists a vertex $v \in V \backslash W$ with $|N^+(v) \cap W| = 1$.  So, either $\left( V(P) \cup V(P_1) \cup   V(P_2) \right)  \cap W = \emptyset$, or $V(P) \cap W >  \lfloor \frac{d(i,j)}{2} \rfloor$.

If $V(P) \cup V(P_1) \cup   V(P_2) = V(D)$, then we have shown that that $V(P) \cap W >  \lfloor \frac{d(i,j)}{2} \rfloor$, since  otherwise $W = \emptyset$, violating the definition of a critical set.  If $V(D) \backslash \left( V(P) \cup V(P_1) \cup   V(P_2) \right) \neq \emptyset$, let $P'$ be the weak path from $v_{x}$ to $v_{y}$ whose internal vertices are exactly those vertices in $V(D) \backslash \left( V(P) \cup V(P_1) \cup   V(P_2) \right)$. Then we can choose $i' \in V_- \cap V(P')$ and $j' \in V_+ \cap V'$ such that $d(i,j)$ is minimal (note that $i'=y, j'=x$ satisfies $i' \in V_- \cap V(P')$ and $j' \in V_+ \cap V'$, so there exists such a minimal $i'$ and $j'$).  We can repeat the same argument as above for this set of vertices, giving us that either $V(P'')  \cap W = \emptyset$, or $V(P') \cap W >  \lfloor \frac{d(i',j')}{2} \rfloor$ where $P''$ is a nonempty weak path containing $P'$ as a sub-weak-path.  We can do this repeatedly until there are no remaining vertices in $V(D)$, giving us that $W = \emptyset$ or $W >  \lfloor \frac{d(i,j)}{2} \rfloor$ for some $i \in V_-$ and $j \in V_+$.  Since the former violates the definition of critical set and $\ell = d(i,j)$ minimizes the latter, it follows that  $|W| \geq \lfloor \ell / 2 \rfloor + 1$ for any critical set $W$.  

Thus, by Observation \ref{obs:critical}, $\F(D) \geq n - 1 - \lfloor \ell / 2 \rfloor$.  
\end{proof}

We establish that a weak cycle $D$ on $n$ vertices can achieve any value of $\F(D)$ between $0$ and $n-1$ by choosing appropriate arc orientations.  There are two constructions, depending on whether $\F(D) > \frac{n}{2}$ or $\F(D)\leq \frac{n}{2}$.  Examples  are shown in Figures \ref{fig:kunderhalfn}--\ref{fig:koverhalfn}.

\begin{figure}[h!]
\begin{minipage}{3in}
\begin{center}
\begin{tikzpicture}[auto]
\def\Radius{0.8cm}
\def\number{10}
\def\othernumber{9}
\tikzstyle{vertex}=[draw, circle, inner sep=0.8mm]
\tikzset{edge/.style = {->,> = Latex}}
\tikzset{curvededge/.style = {->,> = Latex, bend right}}
\foreach \y  in {9}{
\node (v\y) at ({-360/\number*(\y)}:\Radius) [vertex, label=right:$v_{\y}$]{};}
\node (v0) at (0:\Radius) [vertex, label=right:$v_0$]{};
\foreach \y  in {1}{
\node (v\y) at ({-360/\number*(\y)}:\Radius) [vertex, fill=blue, label=right:$v_{\y}$]{};}
\foreach \y  in {3}{
\node (v\y) at ({-360/\number*(\y)}:\Radius) [vertex, fill=blue, label=below:$v_{\y}$]{};}
\node (v2) at ({-360/\number*(2)}:\Radius) [vertex,  label=below:$v_{2}$]{};
\foreach \y  in {6}{
\node (v\y) at ({-360/\number*(\y)}:\Radius) [vertex, label=left:$v_{\y}$  ]{};}
\foreach \y  in {5}{
\node (v\y) at ({-360/\number*(\y)}:\Radius) [vertex, fill=blue, label=left:$v_{\y}$  ]{};}
\node (v4) at ({-360/\number*(4)}:\Radius) [vertex, label=left:$v_{4}$]{};
\foreach \y  in {7, 8}{
\node (v\y) at ({-360/\number*(\y)}:\Radius) [vertex, label=above:$v_{\y}$]{};}
\foreach \y [evaluate=\y as \x using {int(\y+1)}] in {1,3,5,6,7,8}{
\draw[edge] (v\y) to (v\x);}

\foreach \y [evaluate=\y as \x using {int(\y+1)}] in {0, 2, 4}{
\draw[curvededge] (v\y) to (v\x);
\draw[curvededge] (v\x) to (v\y);}
\draw[edge] (v9) to (v0);
\end{tikzpicture}
\caption{Weak cycle with $n=10$,  $F(D)=3$.  Note $V_ - = \emptyset$.}
\label{fig:kunderhalfn}
\end{center}
\end{minipage}
\hskip0.2in
\begin{minipage}{3in}
\begin{center}
\begin{tikzpicture}[auto]
\def\Radius{0.8cm}
\def\number{10}
\def\othernumber{9}
\tikzstyle{vertex}=[draw, circle, inner sep=0.8mm]
\tikzset{edge/.style = {->,> = Latex}}
\tikzset{curvededge/.style = {->,> = Latex, bend right}}

\node (v0) at (0:\Radius) [vertex, label=right:$v_0$]{};
\node (v1) at ({-360/\number*(1)}:\Radius) [vertex, fill=blue, label=right:$v_1$]{};

\foreach \y  in {2, 3}{
\node (v\y) at ({-360/\number*(\y)}:\Radius) [vertex,fill=blue, label=below:$v_{\y}$]{};}

\foreach \y  in {4, 6}{
\node (v\y) at ({-360/\number*(\y)}:\Radius) [vertex, label=left:$v_{\y}$  ]{};}

\node (v5) at ({-360/\number*(5)}:\Radius) [vertex,  fill=blue, label=left:$v_{5}$]{};

\foreach \y  in {8}{
\node (v\y) at ({-360/\number*(\y)}:\Radius) [vertex, label=above:$v_{\y}$]{};}

\foreach \y  in {7}{
\node (v\y) at ({-360/\number*(\y)}:\Radius) [vertex, fill=blue, label=above:$v_{\y}$]{};}

\node (v9) at ({-360/\number*(9)}:\Radius) [vertex, fill=blue, label=right:$v_9$]{};

\draw[edge] (v0) to (v1);
\draw[edge] (v4) to (v3);
\foreach \y [evaluate=\y as \x using {int(\y+1)}] in {1, 2, 4,5, 6, 7, 8}{
\draw[curvededge] (v\y) to (v\x);
\draw[curvededge] (v\x) to (v\y);}
\draw[curvededge] (v9) to (v0);
\draw[curvededge] (v0) to (v9);
\end{tikzpicture}
\caption{Weak cycle with $n=10$, $F(D)=6$. Note $V_0, V_-,$ and $V_+$ are nonempty.}
\label{fig:koverhalfn}
\end{center}
\end{minipage}
\end{figure}

\begin{theorem}
For any $n$ and $k$ with $0\leq k \leq n-1$, there exists a weak cycle $D$ on $n$ vertices such that $\F(D)=k$. \label{thm:alyssa}
\end{theorem}
\begin{proof}
{\bf Case 1: $k \leq \frac{n}{2}$}.
In this case, we use the construction from Theorem \ref{thm:cycleonedirection}.  Let $V_0 = \{0, 2, \ldots 2k-2\}$.  Let $V_+$ consist of all remaining indices between $0$ and $n-1$. 
An example with $n=10$ and $k=3$ is shown in Figure \ref{fig:kunderhalfn}.  By Theorem  \ref{thm:cycleonedirection}, $\F(D) = \sum_{i=1}^{k} \left\lceil \frac{n_i}{2} \right\rceil$.  Note that using $k=0$ runs of vertices from $V_0$ gives us a directed cycle, the special case $\F(D) = 0$.  

{ \bf Case 2: $\frac{n}{2} <k \leq n-1$}.
In this case, we use the construction from Theorem \ref{thm:cyclebothdirections}. Let $V_+=\{ 0 \}$, let $V_- = \{2k-n+1\}$, and let $V_0$ consist of all remaining indices between $0$ and $n-1$.   Then by Theorem  \ref{thm:cyclebothdirections}, 
$\F(D) = n -1 - \left\lfloor \frac{\ell}{2} \right\rfloor = n -1 - \left\lfloor \frac{n-(2k-n+1)}{2} \right\rfloor   = k.$
Note that this includes the case $\F(D)=n-1$, where $v_0$ is a source.  
\end{proof}

We can restate Theorem \ref{thm:alyssa} in terms of critical sets by applying Proposition \ref{prop:charnminusk}.
\begin{corollary} 
For any $n$ and $k$ with $1 \leq k \leq n$, there exists a weak cycle $D$ on $n$ vertices whose smallest critical set is of cardinality $k$.  
\end{corollary}

\subsection{Line graphs including de Bruijn and Kautz digraphs}
\label{sec:line}
We now look at $\F(D)$ in the case that $D$ is an iterated line graph.   Since some of the digraphs we consider here have loops, we describe the following modified CCR that applies only to digraphs with loops.  We continue to use the CCR introduced earlier if $D$ has no loops.
\begin{itemize}
\item $B^0(S) := S$
\item $B^{i+1}(S) := B^i(S) \cup \{w : \{w\} = N^+(v) \backslash B^i(S) \mbox{ for some } v \in V\}$
\end{itemize}
We could also state the CCR as follows: if any vertex $v \in V$ has exactly one empty out-neighbor $u$, then $u$ will be filled.  In other words, the CCR that applies to digraphs with loops is identical to the CCR that applies to digraphs without loops, except that a vertex $u$ may become filled if $u$ is the unique empty out-neighbor of {\bf any} vertex $v$, whether or not $v$ is filled.  We  make the following observation, similar to Observation \ref{obs:critical} but for digraphs with loops.
\begin{obs}
In a digraph with loops, $W$ is a strongly critical set if and only if $V \backslash W$ is a stalled zero forcing set.  
\end{obs}
Note that although $\F(D)$ is defined for any digraph without loops, there exist $D$ with loops that have $\Z(D) = 0$ and therefore $\F(D)$ is undefined.  Also note that if $D$ contains a loop $uu$, then $u \in N^+(u)$ and $u \in N^-(u)$.  The digraph consisting of $V=\{u\}$ and $A=\{ uu \}$ then has $\Z(D)=0$ and $\F(D)$ undefined.  We do not characterize $D$ with $\F(D)$ undefined here, but note the following observation and lemma.  

\begin{obs}
If $D$ is a digraph with loops, then $\F(D)$ is undefined if and only if $\Z(D)=0$.  
\end{obs}

\begin{lemma}
If $N^-(v)=0$ for some $v \in V$, or $N^+(v) \geq 2$ for all $v \in V$ and $V \neq \emptyset$, then $\F(D)$ is defined. \label{lem:defined}
\end{lemma}
\begin{proof}
If $N^-(v)=0$ for some $v \in V$, then $v \in S$ for any ZFS $S$.  Thus, $\Z(D) \geq 1$, and $\F(D)$ is defined.  If $N^+(v) \geq 2$  for all $v \in V$, and $|V| \geq 1$, then each $u \in V$ has at least two out-neighbors, so $\Z(D) \geq 1$ and $\F(D)$ is defined.  
\end{proof}

\begin{dfn}
For a digraph $D = (V,A)$, the \emph{line digraph} of $D$ is the digraph $L(D)$ where 
\begin{itemize}
\item $V(L(D)) = A(D)$, and 
\item $A(L(D)) = \{ab : a, b  \in V(L(D)), \mbox{ and the head of } a \mbox{ is the tail of } b \mbox{ in } D\}$.  
\end{itemize}
\end{dfn}

We use the following result from  \cite[Lemma 3.5]{ferrero2019zero}.
\begin{lemma} \cite{ferrero2019zero}
Let $D$ be a digraph and let $uv$ be a vertex of $L(D)$.  If $deg^+_{L(D)}(uv) \geq 2$, then every subset $T \subseteq N_{L(D)}^+(uv)$ with $|T| \geq 2$ is a strongly critical set in $L(D)$.   \label{lem:ferrero}
\end{lemma}

\begin{obs}
For any weakly connected digraph $D$, $L(D)$ has a source vertex if and only if $D$ has a source vertex.  \label{obs:sourceline}
\end{obs}
We note the following proposition, analogous to Proposition \ref{prop:charnminusk} but for graphs with loops.  
\begin{prop}
In a digraph $D$ with loops, $\F(D) = n-k$ if and only if the minimum cardinality of any strongly critical set in $D$ is $k$.   \label{prop:charnminuskloops}
\end{prop}

\begin{lemma}
Suppose a digraph $D$ with $|V(D)| \geq 2$ and no source has $\deg^+(v) \leq 1$ for all $v \in V(D)$.  Then $D$ is a set of vertex-disjoint  directed cycles.  \label{lem:degreecycles}
\end{lemma}
\begin{proof}
We have  $|V(D)| \leq  \sum_{v\in V} \deg^-(v) = \sum_{v\in V} \deg^+(v)   \leq |V(D)|$.  Hence, $\deg^-(v) = \deg^+(v) = 1$ for every $v \in V(D)$, giving us that $D$ is a set of vertex-disjoint directed cycles.  In particular, if $D$ is weakly connected, then $D$ is a directed cycle.  
\end{proof}

The following theorem establishes $\F(L(D))$, with the added assumption that $\F(L(D))$ is defined in the case that $D$ has loops. 

\begin{theorem}
For any weakly connected digraph $D$ with $|V(D)| \geq 2$, set $m = |A(D)|$.  If $D$ does not have loops, or if $D$ has loops and $\Z(L(D))>0$, then 
$$\F(L(D)) = 
\begin{cases}
0 & \mbox{ if } D \mbox{ is a directed cycle, } \\
 m-1 &  \mbox{ if } D \mbox{ has a source, } \\
m -2 &  \mbox{ otherwise.  }\
\end{cases} 
$$

\end{theorem}
\begin{proof}
If $D$ is a directed cycle, then $L(D)$ is as well, and we know that $\F(L(D))=0$ from Theorem \ref{thm:f0}.
Suppose $D$ has a source. Then by Observation \ref{obs:sourceline}, $L(D)$ has a source $uv$, and $\{uv\}$ forms a critical set in $L(D)$, giving us $\F(L(D)) = |V(L(D))|-1=m-1$.  

Finally, assume that $D$ is not a directed cycle and does not have a source. Then $L(D)$ does not have a source, and is not a directed cycle.  Since $D$ and therefore $L(D)$ are weakly connected, by Lemma \ref{lem:degreecycles}, there exists a vertex $uv \in V(L(D))$ such that $\deg_{L(D)}^+(uv) \geq 2$.   By Lemma \ref{lem:ferrero},  $S=\{xy, wv\}$ is a strongly critical set for any $xy, wv \in N_{L(D)}^+(uv)$.  Then, recalling that $L(D)$ does not have a source, by Proposition  \ref{prop:charnminusk}  or   \ref{prop:charnminuskloops} depending on whether $D$ has loops, $\F(L(D)) = |V(L(D)|-2 = m -2$.  \end{proof}

Two digraph families that can each be defined iteratively  using line digraphs and that are used in multiple applications are de Bruijn and  Kautz digraphs.  See \cite{kaashoek2003koorde, pevzner2001eulerian} for examples of the de Bruijn digraph and \cite{li2004graph, shen2015kautz} for examples of the Kautz digraph in applications.  

For integers $d \geq 2$ and $M \geq 1$, the \emph{de Bruijn digraph} $B(d,M)$ is defined to be the digraph with $V(B(d,M)) = \{x_0 x_1 \ldots x_{M-1} : x_i \in \mathbb{Z}_d \}$, and $A(B(d,M)) = \{(x_0 x_1 \ldots x_{M-1}, x_1 x_2 \ldots x_M)\}$.  The  \emph{Kautz digraph} $K(d,M)$ is defined to be the digraph with $V(K(d,M)) = \{x_0 x_1 \ldots x_{M-1} : x_i \in \mathbb{Z}_{d+1} \mbox{ and } x_i \neq x_{i+1} \}$, and $A(K(d,M)) = \{(x_0 x_1 \ldots x_{M-1}, x_1 x_2 \ldots x_M)\}$.    
Each Kautz digraph and each de Bruijn digraph has vertices of out-degree at least $2$, leading to the following corollary.

\begin{corollary}
If $D$ is a de Bruijn or a Kautz digraph, then $\F(D) = |V(D)|-2$.    
\end{corollary}

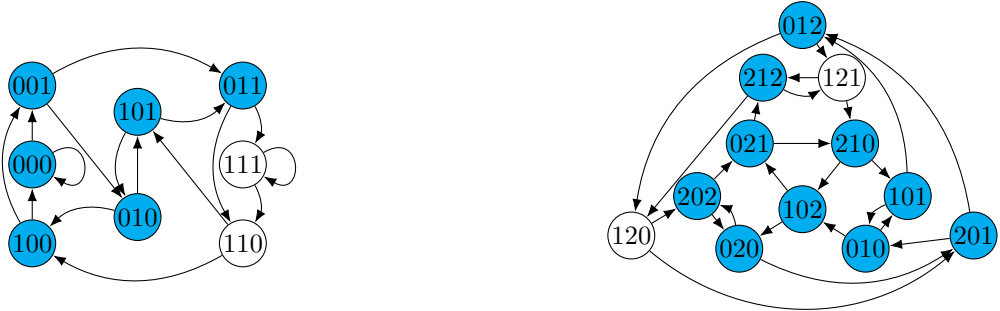
\begin{figure}[h!]
\begin{minipage}{0.475\textwidth}
\begin{center}
\begin{tikzpicture}[->, >= Latex,auto, scale=0.7]
\tikzstyle{vertex}=[draw, circle, inner sep=0.2mm]
\tikzset{edge/.style = {->,> = Latex}}
\tikzset{curvededge/.style = {->,> = Latex, bend right}}
\tikzset{lcurvededge/.style = {->,> = Latex, bend left}}
\node (v001) at (-2, 1.5)[vertex, fill=cyan]{001};
\node (v000) at (-2, 0)[vertex, fill=cyan]{000};
\node (v100) at (-2, -1.5)[vertex, fill=cyan]{100};
\node (v101) at (0,1)[vertex, fill=cyan]{101};
\node (v010) at (0, -1)[vertex, fill=cyan]{010};
\node (v011) at (2, 1.5)[vertex, fill=cyan]{011};
\node (v111) at (2, 0)[vertex]{111};
\node (v110) at (2, -1.5)[vertex]{110};
\draw [edge] (v001) to (v010);
\draw [lcurvededge] (v001) to (v011);
\draw [curvededge] (v101) to (v010);
\draw [curvededge] (v101) to (v011);

\draw [curvededge] (v011) to (v110);
\draw [lcurvededge] (v011) to (v111);
\draw [lcurvededge] (v111) to (v110);
\draw (v111) to [edge, out=30, in=330, looseness=5] (v111);  %loop

\draw [curvededge] (v010) to (v100);
\draw [edge] (v010) to (v101);
\draw [lcurvededge] (v110) to (v100);
\draw [edge] (v110) to (v101);

\draw [edge] (v000) to (v001);

\draw (v000) to [edge, out=30, in=330, looseness=5] (v000);  %loop

\draw [edge] (v100) to (v000);
\draw [lcurvededge] (v100) to (v001); 
\end{tikzpicture}
\end{center}
\end{minipage}
\hfill
\begin{minipage}{0.475\textwidth}
\begin{center}
\begin{tikzpicture}[auto, scale=0.7]
\tikzstyle{vertex}=[draw, circle, inner sep=0.2mm]
\tikzset{edge/.style = {->,> = Latex}}
\tikzset{curvededge/.style = {->,> = Latex, bend right}}
\tikzset{reallycurvededge/.style = {->,> = Latex, in=220, out=320}}
\node (v012) at (0,3)[vertex, fill=cyan]{012};
\node (v212) at (-0.75, 2)[vertex, fill=cyan]{212};
\node (v121) at (0.75, 2)[vertex]{121};
\node (v021) at (-1, 0.75)[vertex, fill=cyan]{021};
\node (v210) at (1,0.75)[vertex, fill=cyan]{210};
\node (v202) at (-2,-0.25)[vertex, fill=cyan]{202};
\node (v101) at (2,-0.25)[vertex, fill=cyan]{101};
\node (v102) at (0,-0.5)[vertex, fill=cyan]{102};
\node (v120) at (-3.25,-1)[vertex]{120};
\node (v020) at (-1.2,-1.25)[vertex, fill=cyan]{020};
\node (v010) at (1.2,-1.25)[vertex, fill=cyan]{010};
\node (v201) at (3.25,-1)[vertex, fill=cyan]{201};
\draw[curvededge](v012)to(v120);
\draw[edge](v012)to(v121);
\draw[edge](v010)to(v101);
\draw[edge](v010)to(v102);
\draw[edge](v021)to(v210);
\draw[edge](v021)to(v212);
\draw[curvededge](v020)to(v201);
\draw[curvededge](v020)to(v202);
\draw[edge](v212)to(v120);
\draw[curvededge](v212)to(v121);
\draw[edge](v210)to(v101);
\draw[edge](v210)to(v102);
\draw[edge](v202)to(v020);
\draw[curvededge](v201)to(v012);
\draw[edge](v201)to(v010);
\draw[edge](v202)to(v021);
\draw[edge](v121)to(v210);
\draw[edge](v121)to(v212);
\draw[reallycurvededge](v120)to(v201);
\draw[edge](v120)to(v202);
\draw[curvededge](v101)to(v010);
\draw[curvededge](v101)to(v012);
\draw[edge](v102)to(v020);
\draw[edge](v102)to(v021);
\end{tikzpicture}
\end{center}
\end{minipage}
\caption{de Bruijn digraph $B(2,3)$ and Kautz digraph $K(2,3)$ with FZFS in cyan}
\label{fig:minimalmaximal}
\end{figure}

\section{Open problems}
\label{sec:openproblems}
Since computing $\F(G)$ for undirected graphs was found to be NP-hard in \cite{shitov2017complexity}, it follows that the same is true for digraphs.  However, it is unknown whether or not this remains true if we restrict to oriented graphs.  Indeed, at the time this paper was written, this result had not been established for the zero forcing number $\Z(\overrightarrow{G})$ where $\overrightarrow{G}$ is an oriented graph.  

While we considered line digraphs with loops in Section \ref{sec:line}, more general investigation of $\F(D)$ in the case $D$ has loops would be interesting.  In particular, a characterization of digraphs with loops that have $\F(D)$ undefined (and therefore $\Z(D)=0$) is a possible starting point.  

We can also consider the following generalization of this problem.  A ZFS $S$ is a \emph{minimal} ZFS if deleting any vertex from $S$ results in the new set being a FZFS.  Similarly, a FZFS $S$ is a \emph{maximal} FZFS if adding any vertex to $S$ results in the new set being a ZFS.  Certainly, any minimum ZFS is also minimal, and any maximum FZFS is also maximal.  However, for some digraphs there exist examples of minimal ZFS and maximal FZFS that are not minimum and not maximum respectively, as in Figure \ref{fig:minimalmaximal}.

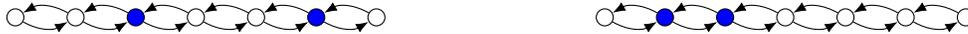
\begin{figure}[h!]
\begin{minipage}{0.475\textwidth}
\begin{center}
\begin{tikzpicture}[auto]
\tikzstyle{vertex}=[draw, circle, inner sep=0.8mm]
\tikzset{edge/.style = {->,> = Latex}}
\tikzset{curvededge/.style = {->,> = Latex, bend right}}
\foreach \x in {3, 6}{
\node (v\x) at (0.8*\x,0)[vertex, fill=blue]{};}
\foreach \x in {1, 2, 4, 5, 7}{
\node (v\x) at (0.8*\x,0)[vertex]{};}
\foreach \x  [evaluate=\x as \y using {int(\x+1)}]  in {1, 2, 3, 4, 5, 6}{
\draw[curvededge](v\x)to(v\y);
\draw[curvededge](v\y)to(v\x);}
\end{tikzpicture}
\end{center}
\end{minipage}
\begin{minipage}{0.475\textwidth}
\begin{center}
\begin{tikzpicture}[auto]
\tikzstyle{vertex}=[draw, circle, inner sep=0.8mm]
\tikzset{edge/.style = {->,> = Latex}}
\tikzset{curvededge/.style = {->,> = Latex, bend right}}
\foreach \x in {2, 3}{
\node (v\x) at (0.8*\x,0)[vertex, fill=blue]{};}
\foreach \x in {1, 4, 5, 6, 7}{
\node (v\x) at (0.8*\x,0)[vertex]{};}
\foreach \x  [evaluate=\x as \y using {int(\x+1)}]  in {1, 2, 3, 4, 5, 6}{
\draw[curvededge](v\x)to(v\y);
\draw[curvededge](v\y)to(v\x);}
\end{tikzpicture}
\end{center}
\end{minipage}
\caption{A maximal FZFS that is not maximum and a minimal ZFS that is not minimum}
\label{fig:minimalmaximal}
\end{figure}

Let $Z_m(D)$ denote the set of minimal ZFS of $D$, and let $F_M(D)$ denote the set of maximal FZFS of $D$.  If $\F(D) < \Z(D)$, then $F_M(D)$ is precisely the set of maximum FZFS, and $Z_m(D)$ is precisely the set of minimum ZFS.  However, it would be interesting to study these parameters for digraphs with $\F(D) \geq \Z(D)$.  For example, we can ask which integers $k$ with $0 < k < n$ have the property that there exists a ZFS $S \in Z_m(D)$ with $|S| = k$.  We can ask the analogous question for maximal FZFS as well.  

\bigskip

 \bibliography{directedgraphsbib}

\end{document}